\documentclass[11pt]{amsart}
\usepackage[top=1in, bottom=1in, left=1.3in, right=1.3in]{geometry}
\usepackage{amsmath}
\usepackage{amssymb}
\usepackage{amsthm}
\usepackage{verbatim}
\usepackage{graphicx}
\usepackage{graphics}
\usepackage{color}
\usepackage{enumerate}
\usepackage{mathrsfs}
\usepackage{booktabs}
\usepackage[toc,page]{appendix}
\usepackage{enumerate}
\usepackage{mathrsfs}
\usepackage{fancyhdr}
\usepackage{amsopn, amstext, amscd,pifont}
\usepackage{amssymb,bm, cite}
\usepackage[all]{xy}
\usepackage{color}
\usepackage{graphicx,hyperref}
\usepackage{float}
\usepackage{listings}
\usepackage{setspace}
\usepackage{url}
\usepackage{paralist}
\usepackage{subcaption}

\numberwithin{equation}{section}

\usepackage{graphicx,pgf}
\usepackage{float}

\usepackage{fancyhdr}

\theoremstyle{plain}

\newtheorem{theorem}{Theorem}[section]
\newtheorem{proposition}[theorem]{Proposition}
\newtheorem{lemma}[theorem]{Lemma}
\newtheorem{corollary}[theorem]{Corollary}

\theoremstyle{definition}
\newtheorem{definition}[theorem]{Definition}

\theoremstyle{remark}
\newtheorem{example}[theorem]{Example}

\graphicspath{{images/}}

\begin{document}
\title[Nonarchimedean Lyapunov exponents of polynomials]{Nonarchimedean Lyapunov exponents of polynomials}
\author{Hongming Nie}
\address{Institute for Mathematical sciences, Stony Brook University}
\email{hongming.nie@stonybrook.edu}

\maketitle

\begin{abstract}
Let $K$ be an algebraically closed and complete nonarchimedean field with characteristic $0$ and let $f\in K[z]$ be a polynomial of degree $d\ge 2$. We study the Lyapunov exponent $L(f,\mu)$ of $f$ with respect to an $f$-invariant and ergodic Radon probability measure $\mu$  on the Berkovich Julia set of $f$ and the lower Lyapunov exponent $L_f^{-}(f(c))$ of $f$ at a critical value $f(c)$. Under an integrability assumption, we show $L(f,\mu)$ has a lower bound only depending on $d$ and $K$. In particular, if $f$ is tame and has no wandering nonclassical Julia points, then $L(f,\mu)$ is nonnegative; moreover, if in addition $f$ possesses a unique Julia critical point $c_0$, we show $L_f^{-}(f(c_0))$ is also nonnegative.

\end{abstract}


\section{Introduction}\label{sec:intro}

Let $K$ be an algebraically closed field with characteristic $0$  that is complete with respect to a nonarchimedean and nontrivial absolute value $|\cdot|$. 
Consider the residue characteristic $\mathrm{res. \ char}(K)$ of $K$ and set 
\begin{equation}\label{q}
p=p_K:=
\begin{cases}
 \mathrm{res. \ char}(K)\ \ &\text{if}\  \mathrm{res. \ char}(K)>0,\\
e\ \ &\text{if}\  \mathrm{res. \ char}(K)=0.
\end{cases}
\end{equation}

Let  $\phi\in K(z)$ be a rational map of degree $d\ge 2$ and denote by $\phi^\#$ the chordal derivative of $\phi$ on the projective space $\mathbb{P}^1:=\mathbb{P}^1_{K}$ with respect to the chordal metric. Both $\phi$ and $\phi^\#$ extend to the Berkovich space $\mathsf{P}^1:=\mathsf{P}^1_K$ over $K$ by continuity, and hence to the Berkovich hyperbolic space $\mathsf{H}^1:=\mathsf{P}^1\setminus\mathbb{P}^1$. For each $\xi\in\mathsf{P}^1$, the \emph{lower Lyapunov exponent} $L^-_\phi(\xi)$ of $\phi$ at $\xi$ is defined as
\begin{equation}\label{L-}
L^-_\phi(\xi):=\liminf_{n\to\infty}\frac{1}{n}\sum_{k=0}^{n-1}\log_p \phi^\#(\phi^k(\xi)),
\end{equation}
and \emph{pointwise Lyapunov exponent} $L_\phi(\xi)$ of $\phi$ at $\xi$ (if exists) is defined as 
\begin{equation*}\label{L}
L_\phi(\xi):=\lim_{n\to\infty}\frac{1}{n}\sum_{k=0}^{n-1}\log_p \phi^\#(\phi^k(\xi)),
\end{equation*}

Let $\nu$ be a $\phi$-invariant Radon probability measure on the (Berkovich) Julia set $J(\phi)$ of $\phi$ such that $\log_p\phi^\#$ is $\nu$-integrable. Then the \emph{Lyapunov exponent} of  $\phi$ with respect to $\nu$ is defined as
\begin{equation}\label{L}
L(\phi,\nu):=\int_{\mathsf{P}^1}\log_p\phi^\#d\nu.
\end{equation}
If in addition that $\nu$ is ergodic, the Birkhoff ergodic theorem implies that  $L_\phi(\xi)$ exists and coincides to $L(\phi,\nu)$ for $\nu$-almost all points $\xi\in\mathsf{P}^1$.

In this paper, we mainly consider the Lyapunov exponents for polynomials, and in particular, we study the lower Lyapunov exponents at critical values.



\subsection{Statement of main results}\label{sec:results}

For any $n\ge 1$, set 
$$\kappa_n:=\min\{\log_p|\ell|:1\le\ell\le n\}$$
 and note that $\kappa_n\le 0$. For a polynomial $f\in K[z]$ of degree $d\ge 2$, we set 
 $$\kappa(f):=\min\{\log_p|\deg_\xi f|:\xi\in\mathsf{P}^1\},$$ 
 then $\kappa_d\le\kappa(f)\le 0$. Moreover, we also denote by $J_I(f):=J(f)\cap K$ the \emph{classical Julia set} of $f$.

For polynomials, we have the following lower bound for the Lyapunov exponent. 
For $\xi\in\mathsf{H}^1$, denote by $\mathrm{diam}(\xi)$ the diameter of $\xi$, see Section \ref{sec:berk}.

\begin{theorem}\label{thm:bound}
Let $f\in K[z]$ be a polynomial of degree $d\ge 2$ and let $\mu$ be an $f$-invariant Radon probability measure on $J(f)$. Then there exists a constant $C(f)>0$ such that the following hold.
\begin{enumerate}
\item If $\mu(J_I(f))>0$, then for $\mu$-almost point $\xi$ in $J_I(f)$,
$$\kappa(f)\le L_f(\xi)\le C(f)\mu(J_I(f));$$  
in particular, $\log_pf^\#$ is $\mu$-integrable on $J_I(f)$.
\item If $\mu(J(f)\cap\mathsf{H}^1)>0$ and $\log_p\mathrm{diam}(\cdot)$ is $\mu$-integrable on $J(f)\cap\mathsf{H}^1$, then 
$$\kappa(f)\le \int_{J(f)\cap\mathsf{H}^1}\log_pf^\#d\mu\le C(f)\mu(J(f)\cap\mathsf{H}^1).$$
\end{enumerate}
 Moreover, if in addition $\mu$ is ergodic and either $\mu(J_I(f))=1$ or  $\log_p\mathrm{diam}(\cdot)$ is $\mu$-integrable on $J(f)\cap\mathsf{H}^1$, 
$$L(f,\mu)\ge\kappa(f)\ge\kappa_d.$$
\end{theorem}

 A dynamically meaningful invariant  and ergodic measure is the equilibrium measure $\mu_f$ which satisfies $f^\ast\mu_f=d\cdot\mu_f$, see \cite[Theorem 0.1]{Favre04}. Jacobs have proven that   $L(f,\mu_f)\ge\kappa_d$, see \cite[Theorem 1]{Jacobs19}. In fact, by adding an extra term involving with Lipschitz constants, Jacobs has shown a lower bound in the case of rational maps; for polynomials, we can omit this extra term by the work of Rumely-Winburn \cite[Theorem 0.3]{Rumely15} on Lipschitz constants. Since either $\mu_f(J_I(f))=1$ or $\mathrm{supp}(\mu_f)\subset\mathsf{H}^1$ (and hence $\log_p\mathrm{diam}(\cdot)$ is $\mu_f$-integrable on $J(f)$), our Theorem \ref{thm:bound} extends Jacobs' result on the equilibrium measure in the polynomial case.

There are two main advantages in our argument for considering polynomials rather than general rational maps: first, the derivative of a polynomial is bounded in any fixed disk in $K$, specially in a disk intersecting classical Julia set if such Julia set is nonempty; and second, under any polynomial, the iterated preimages of any disk are all disks. Although we do not have an explicit example in hand, due to the noncompactness of the classical Julia sets in general setting, it seems improbable to achieve the bounded derivatives for arbitrary rational maps. For the preimages part, under a rational map, an open disk may have preimages which are not disks, see \cite[Section 3.6]{Benedetto19}. Such advantages play a key role: the bounded derivative can bound the number of critical values of the iterations in a certain neighborhood of a critical point, while the disk preimages relate the derivatives and the ratio of diameters of certain disks. See Section \ref{sec:derivative}  for details.

\medskip

The \emph{ramification locus} $\mathcal{R}_f$ of $f$ is the set of points $\xi$ in $\mathsf{P}^1$ with local degree $\deg_{\xi}f$ at least $2$; for the geometry and topology of $\mathcal{R}_f$, see \cite{Faber13I} and \cite{Faber13II}.  We say $f$ is \emph{tame} if $\mathcal{R}_f$ is locally finite, in this case $p$ does not divide the local degree $\deg_\xi f$ at any $\xi\in\mathsf{P}^1$; otherwise, we say $f$ is \emph{nontame}. We can upgrade Theorem \ref{thm:main} as following for tame polynomials. 
\begin{corollary}\label{cor:bound}
Let $f\in K[z]$ be a tame polynomial of degree at least $2$ and let $\mu$ be an $f$-invariant Radon probability measure on $J(f)$. Then the following hold.
\begin{enumerate}
\item If $\mu(J_I(f))>0$, then for $\mu$-almost point $\xi$ in $J_I(f)$,
$$L_f(\xi)\ge 0.$$  
\item If $\mu(J(f)\cap\mathsf{H}^1)>0$ and $\log_p\mathrm{diam}(\cdot)$ is $\mu$-integrable on $J(f)\cap\mathsf{H}^1$, then 
$$\int_{J(f)\cap\mathsf{H}^1}\log_pf^\#d\mu=0.$$
\end{enumerate}
Moreover, if in addition $\mu$ is ergodic and either $\mu(J_I(f))=1$ or  $\log_p\mathrm{diam}(\cdot)$ is $\mu$-integrable on $J(f)\cap\mathsf{H}^1$, 
$$L(f,\mu)\ge 0.$$
\end{corollary}

By a result of Trucco \cite[Section 3]{Trucco14}, the tameness assumption in Corollary \ref{cor:bound} implies that either $J(f)$ is a singleton in $\mathsf{H}^1$ or $J_I(f)$ is nonempty. The upgraded lower bounds in Corollary \ref{cor:bound} apply the observation that $\kappa(f)=0$, and the one in Corollary \ref{cor:bound} (2) also employs an improved distortion between derivatives and ratios of certain disks, see Propositon \ref{prop:distortion}. Obviously  if $f$ has a repelling fixed point in $J_I(f)$ with whole mass under $\mu$, then $L(f,\mu)>0$. For an example with $\mu(J_I(f))=0$,  we can take $f$ with a fixed point in $J(f)\cap\mathsf{H}^1$ and let $\mu$ be the Dirac measure at this fixed point. 

We say a Julia point in $\mathsf{H}^1$ is \emph{wandering} if it is not eventually periodic. For polynomials without wandering Julia points in $\mathsf{H}^1$, the invariance of $\mu$ implies that $\mathrm{supp}(\mu)\cap\mathsf{H}^1$ consists of periodic points, which guarantees the integrality assumption. Thus in this case, we have the following.

\begin{corollary}\label{cor:bound1}
Let $f\in K[z]$ be a tame polynomial of degree at least $2$ with no wandering Julia points in $\mathsf{H}^1$ and let $\mu$ be an $f$-invariant and ergodic Radon probability measure on $J(f)$. Then
$$L(f,\mu)\ge 0.$$
\end{corollary}
For certain fields $K$, there exist tame polynomials in $K[z]$ having wandering Julia points in $\mathsf{H}^1$, see \cite[Theorem A]{Trucco15}. However, if $K$ contains a discretely valued field whose algebraic closure is dense in $K$, then for any tame polynomial in $K[z]$ of degree at least $2$, its Julia points in $\mathsf{H}^1$ coincide with the grand orbits of finitely many periodic orbits \cite[Theorem A]{Trucco14}, and hence $\mathsf{H}^1$ contains no wandering Julia points. The $p$-adic field $\mathbb{C}_p$ and the completion $\mathbb{L}$ of the Puiseux series over $\mathbb{C}$ provide paradigms of such above field $K$. We should mention here that even for such a field $K$, there exist nontame polynomials possessing wandering Julia points, see \cite[Theorem 1.1]{Benedetto02}. For more references about the (non)wandering Julia points, we refer \cite{Benedetto00} and \cite{Rivera05}.

\medskip
 Under the assumptions in Corollary \ref{cor:bound1}, for $\mu$-almost all points $\xi\in J(f)$, the Birkhoff ergodic theorem implies that the lower Lyapunov exponent $L_f(\xi)\ge 0$. It is natural to ask if the critical values of $f$ also have nonnegative (lower) Lyapunov exponents.  In the case of several Julia critical points, there exists a quartic tame polynomial having a Julia critical value whose lower Lyapunov exponent is $-\infty$, see Example \ref{ex:2}. In the case of unique Julia critical point, we obtain the following nonnegativeness property.  


\begin{theorem}\label{thm:main}
Let $f\in K[z]$ be a tame polynomial of degree at least $2$ with no wandering Julia points in $\mathsf{H}^1$. If $f$ has a unique critical point $c$ in $J_I(f)$, then there exists $\alpha>0$ such that for all $n\ge 1$, 
$$|(f^n)'(f(c))|\ge p^{-\alpha}.$$
In particular, 
$$L_f^{-}(f(c))\ge 0.$$
\end{theorem}

Corollary \ref{cor:>} states a sufficient condition for $L_f^{-}(f(c))>0$. On the other hand, Example \ref{ex:0} contains a cubic tame polynomial having a unique Julia critical value whose lower Lyapunov exponent is $0$; this example is provided by J. Kiwi.


\medskip
The no wandering Julia points in $\mathsf{H}^1$ in Theorem \ref{thm:main} implies that any component of (Berkovich) Fatou $F(f):=\mathsf{P}^1\setminus J(f)$ is either (pre)periodic or contained in the basin of a type II (repelling) cycle, see Lemma \ref{lem:nonwandering}. To supplement Theorem \ref{thm:main}, for the Fatou critical points, we have

\begin{proposition}\label{prop:main}
Let $f\in K[z]$ be a tame polynomial of degree at least $2$ with no wandering Julia points in $\mathsf{H}^1$.  Suppose that $c\in K$ is a critical point contained in $F(f)$. If $c$ is not in the basin of an attracting cycle of $f$, then there exists $\ell_0\ge 1$ such that for any $\ell\ge\ell_0$,
$$L_f(f^\ell(c))=0.$$
\end{proposition}

The critical point $c$ in Proposition \ref{prop:main} lies in either the basin of an indifferent cycle or the basin of a type II repelling cycle, see Section \ref{sec:dynamics}. We can choose $\ell_0\ge 1$ such that  the orbit of the Fatou component containing $f^{\ell_0}(c)$ is disjoint with all critical points of $f$.

\subsection{Motivation and background}
For any rational map $\phi\in K(z)$ of degree at least $2$, the equilibrium measure is the unique $\phi$-invariant probability measure $\mu_\phi$ such that $\phi^\ast\mu_\phi=d\cdot\mu_\phi$, see \cite[Theorem 0.1]{Favre04}. Although $\mu_\phi$ may not be the measure of maximal entropy (see \cite[Section 5.2]{Favre10}), as mentioned in Section \ref{sec:results}, the Lyapunov exponent $L(\phi,\mu_\phi)$ already appears in literatures.  Considering the measures converging to $\mu_\phi$, Okuyama \cite[Lemma 3.1]{Okuyama12} and Jacobs \cite[Corollary 1]{Jacobs18} have independently obtained different approximation of $L(\phi,\mu_\phi)$. Applying the multipliers of periodic points, Okuyama has stated quantitative approximations to $L(\phi,\mu_\phi)$, see \cite[Theorem 1]{Okuyama12} and \cite[Theorems 1 and 2]{Okuyama15}. From Okuyama's results and a work of Benedetto, Ingram, Jones and Levy \cite[Theorem 4.1]{Benedetto14}, it is easy to conclude that $L(\phi,\mu_\phi)\ge 0$ if $\phi$ is tame.  In contrast, if $\phi$ is nontame,  $L(\phi,\mu_\phi)$ may be negative; for example, the map $\phi(z)=z^d$ has $L(\phi,\mu_\phi)=\log_p|d|<0$ if $p\mid d$.  Jacobs has showed a lower bound of $L(\phi,\mu_\phi)$, see \cite[Theorem 1]{Jacobs19}. Instead of a single map, for the analytic families of maps in $K[z]$, Favre and Gauthier \cite[Corollary 1]{Favre18}  and Gauthier, Okuyama and Vigny \cite[Theorem B]{Gauthier20} have studied the Lyapunov exponents with respect to the corresponding equilibrium measures.

All the above known results are only related to the equilibrium measure, which have motivated us to investigate the Lyapunov exponent with respect to an arbitrary invariant Radon measure. In general, such exponents are untransparent due to the obscure of the measures. Our Theorem \ref{thm:bound} starts the study of these exponents for polynomials. 

Once the Lyapunov exponent is at our disposal, we can ask about the (lower)  Lyapunov exponents at dynamically meaningful points. Such exponents at the critical values reflect the asymptotic behaviors of derivatives along critical trajectories. For example, if  $K=\mathbb{C}_p$,  Benedetto \cite{Benedetto01} has introduced the non-archimedean hyperbolic rational maps with coefficients in a finite extension of $\mathbb{Q}_p$ acting on $\mathbb{P}^1$ (and hence on $\mathsf{P}^1$). Such rational maps have no critical points in Julia sets \cite[Main Theorem]{Benedetto01}, and hence in the tame case, the corresponding lower Lyapunov exponent at any critical value is nonpositive; indeed, such a map has no wandering domains in $\mathsf{P}^1$ \cite[Theorem 11.14]{Benedetto19}, which implies, according to the classification of Fatou components \cite[Th\'eor\`eme de Classification]{Rivera03},  that any critical value eventually maps to an indifferent domain or an attracting domain.

Trucco has described the dynamics of polynomials with great details in \cite{Trucco14}. Further study of such dynamics appears in \cite{Nie22} and \cite{Nie20}. Although those dynamics are more tractable than that for general rational maps, as aforementioned, polynomials with several Julia critical points may possess negative lower Lyapunov exponents at critical values. It  attracts our attention to the polynomials with a unique Julia critical point.

In complex dynamics,  for any rational map  of degree at least $2$  any invariant probability measure on the Julia set,  Przytycki has proved that the pointwise Lyapunov exponent is nonnegative for almost all point, see  \cite[Theorem A]{Przytycki93}. Then Levin, Przytycki and Shen have showed that for any unicritical polynomial in $\mathbb{C}[z]$, if the unique critical value does not belong to the basin of an attracting cycle, then the lower Lyapunov exponent at the critical value is nonnegative, see \cite[Theorem 1.1]{Levin16}. Our results in this present paper are nonarchimedean counterparts of the aforementioned results in complex dynamics.

\subsection{Brief overview of the proofs}

We proceed proof of Theorem \ref{thm:bound} as follows. 

If $\mu(J_I(f))>0$, it is an adaptation of  the one given by Przytycki in \cite{Przytycki93} for complex rational maps. The main idea is the following. Pick a generic (in certain sense) point $x$ in the classical Julia set $J_I(f)$. If the orbit of $\xi$ doses not intersect certain small shrinking neighborhoods of the critical points, we can bound $L_f(\xi)$ below by considering a small disk containing $\xi$ and its image under $f^n$. If the orbit of $\xi$ intersects those shrinking neighborhoods of the critical points, we estimates the derivatives of some iteration of $f$, see Proposition \ref{lem:kappa}, which also implies a lower bound of $L-f(\xi)$. Then applying a standard argument, we can upgrade to the desired lower bound. The nonarchimedean features involved make the estimates in the proof more accessible. We mention here that in  \cite[Appendix A]{Rivera20} Rivera-Letelier has given a simpler proof of Przytycki's result in which a distortion for the maps is crucial; in our setting, mainly due to the possible noncompactness of $J_I(f)$, we do not know if such distortion holds for $f$ or not. 

If $\mu(J(f)\cap\mathsf{H}^1)>0$, the conclusion follows from a distortion of diameters of points in $\mathsf{H}^1$, see Proposition \ref{prop:distortion}. This idea has been applied in Jacobs' proof of  \cite[Theorem 2]{Jacobs19}. 

In Jacobs' proof of the lower bound of $L(f,\mu_f)$ (see \cite[Theorem 1]{Jacobs19}) with respect to the equilibrium measure $\mu_f$, it employs equidistribution results about $\mu_f$, see \cite[Section 4.3]{Jacobs19}; however, such equidistribution results are lacking in our setting for general invariant and ergodic measures, so we can not exploit Jacobs' proof directly here.


\smallskip
To prove Theorem \ref{thm:main}, we use several trees to control the map $f^n$ at the unique Julia critical point $c$ of $f$.  One of the trees, called the omitted tree, in this paper concerns the iteration of the ramification locus by deleting parts involving the critical point $c$. Since in our setting the derivatives of $f^n$ are the ratios of diameters of some disks, we track the expansion of such diameters with marked grids. A key point is that, if there is no wandering Julia points in $\mathsf{H}^1$, then the omitted tree is bounded away from $K$, see Proposition \ref{lem:bound}. 

\subsection*{Structure of the paper}
Section \ref{sec:poly} covers the basic background on Berkovich space and the dynamics of polynomials. It also contains estimates of derivatives.   Section \ref{sec:tree} states several trees for polynomials and Section \ref{sec:Lya} provides an approximation of the (lower) Lyapunov exponent. Sections \ref{sec:proof1} and \ref{sec:proof2} prove the results in this paper. Section \ref{sec:ex} contains examples for the Lyapunov exponents at critical values. 

\subsection*{Acknowledgements}
 The author  would like to thank Jan Kiwi for fruitful discussions (specially in the time when the author was in Pontificia Universidad Cat\'olica de Chile) and for providing Example \ref{ex:0}. The author would also like to  thank Y\^usuke Okuyama for helpful comments and Feliks Przytycki for mentioning the reference \cite{Rivera20}.

\section{Polynomial dynamics background}\label{sec:poly}

In this section, we discuss some results from polynomial dynamics on $K$ and on the Berkovich space $\mathsf{P}^1$ which will be used in the remainder of this paper. 

\subsection{Berkovich space}\label{sec:berk}
Recall that $K$ is an algebraically closed filed with characteristic $0$ that is complete with respect to a nonarchimedean and nontrivial absolute value $|\cdot |$. 
The closed disks and the open disks in $K$ are the sets of from 
 $$\overline{D}(x,r):=\{y\in K:|y-x|\le r\}\ \text{and}\  D(x,r):=\{y\in K:|y-x|< r\}$$
 for some $x\in K$ and $r\ge 0$, respectively. The nonarchimedean property of $|\cdot|$ implies that any point in a disk is a center, that is, $\overline{D}(x,r)$ (w.r.t. $D(x,r)$) coincides with $\overline{D}(y,r)$ (w.r.t. $D(y,r)$) for any $y\in\overline{D}(x,r)$ (w.r.t. $y\in D(x,r)$). The diameter of a disk $\overline{D}(x,r)$ or $D(x,r)$ is
 $$\mathrm{diam}(\overline{D}(x,r))=\mathrm{diam}(D(x,r))=r.$$ 
Moreover, any two disks are either disjoint or one contains the other. 

Associated with the topology induced by $|\cdot|$, the space $K$ is totally disconnected and not locally compact, so is its projective space $\mathbb{P}^1$. To remedy this, it is compactifyed by the Berkovich space $\mathsf{P}^1$ over $K$. For the development of $\mathsf{P}^1$, we refer \cite{Berkovich90}.  Here, we state briefly the description of $\mathsf{P}^1$.

The Berkovich affine space $\mathsf{A}^1$ over $K$ consists of the seminorms on the polynomial ring $K[z]$ which extends $|\cdot |$ in $K$. The points $\xi\in\mathsf{A}^1$  correspond to (cofinal equivalence classes of) nested decreasing sequences of closed disks $\overline{D}(x_i,r_i)\subset K$  for some $x_i\in K$ and $r_i> 0$ via the limit of the supermum seminorm on each disk, that is for $\psi\in K[z]$,
$$||\psi||_\xi=\lim_{i\to\infty}\sup_{z\in\overline{D}(x_i,r_i)}|\psi(z)|.$$ 
The space $\mathsf{A}^1$ possibly contains four types points. The type I points, also known as \emph{classical points}, are the points whose corresponding sequence has singleton intersection;  the type II points are those whose corresponding sequence has a disk intersection with diameter in the valued group $|K^\times|$; the type III points are those whose corresponding sequence has a disk intersection with diameter not in $|K^\times|$; and the type IV points are those whose corresponding sequence has empty intersection. In particular, for a type I, II or III point, the above intersection is the disk in $K$ corresponding to this point. The space $\mathsf{A}^1$ always contains type I points and type II points, but may not contain type III or type IV points, see \cite[Section 4]{Faber13I}. The space $K$ is identified with the set of type I points. In general, for type I, II or III points, we write $\xi_{x,r}$ for a point in $\mathsf{A}^1$, where $x\in K$ is a point in the intersection of the aforementioned disks and $r\ge 0$ is the diameter of the intersection, and denote by $\mathrm{diam}(\xi_{x,r})=r$. We call the type II point $\xi_G:=\xi_{0,1}$  corresponding to the unit disk in $K$ is the \emph{Gauss point}.   Moreover, the inclusion between the disks of $K$ induces a natural partial order $\preceq$ in $\mathsf{A}^1$, which induces a tree structure on $\mathsf{A}^1$. 

Since the polynomials in $K[z]$ map disks to disks in $K$, they can extend to $\mathrm{A}^1$. Now we associate $\mathsf{A}^1$ with the weak topology, also know as the Gelfand topology, in which the map $\xi\to f(\xi)$ on $\mathrm{A}^1$ is continuous for all $f\in K[z]$. Then $\mathrm{A}^1$ is Hausdorff,  locally compact and unique arcwise connected. We also remark here that type I points are dense in $\mathsf{A}^1$, so are the points of type II. 

The Berkovich space $\mathsf{P}^1$ over $K$ can be regarded as the one point compactification of $\mathsf{A}^1$ by adding the unique maximal (type I ) element $\infty$ in the partial order $\preceq$. Then associated with the weak topology, $\mathsf{P}^1$ is a Hausdorff, arcwise connected, compact and nonmetrizable topological space with tree structure. For two points $\xi_1,\xi_2\in\mathsf{P}^1\setminus\{\xi\}$, we say $\xi_1\prec\xi_2$ if $\xi_1\preceq\xi_2$ and $\xi_1\not=\xi_2$.  Denote by $[\xi_1,\xi_2]\subset\mathsf{P}^1$ the unique segment connecting $\xi_1$ and $\xi_2$, and denote by $]\xi_1,\xi_2[:=[\xi_1,\xi_2]\setminus\{\xi_1,\xi_2\}$, $]\xi_1,\xi_2]:=[\xi_1,\xi_2]\setminus\{\xi_1\}$ and $[\xi_1,\xi_2[:=[\xi_1,\xi_2]\setminus\{\xi_2\}$.

At a point $\xi\in\mathsf{P}^1$, for two points $\xi_1, \xi_2\in\mathsf{P}^1$, we say $\xi_1$ is equivalent to $\xi_2$ if $]\xi,\xi_1[\ \cap\ ]\xi,\xi_2[\not=\emptyset$. This induces an equivalent relation on $\mathsf{P}^1\setminus\{\xi\}$. Each equivalence class is a \emph{direction} at $\xi$. The \emph{tangent space} $T_\xi\mathsf{P}^1$ consists of all the directions at $\xi$. We denote by $\overrightarrow{\xi\xi_1}\in T_\xi\mathsf{P}^1$ the direction at $\xi$ corresponding to $\xi_1\not=\xi$. If $\xi$ is of type I or IV, then $T_\xi\mathsf{P}^1$ is a singleton; if $\xi$ is of type II, then $T_\xi\mathsf{P}^1$ is naturally isomorphic to the projective space of the residue field of $K$; and if $\xi$ is of type III, then $T_\xi\mathsf{P}^1$ contains two directions. For $\overrightarrow{\xi\xi_1}\in T_{\xi}\mathsf{P}^1$, denote by $\mathsf{B}(\overrightarrow{\xi\xi_1})$ the component of $\mathsf{P}^1\setminus\{\xi\}$ corresponding to $\overrightarrow{\xi\xi_1}$ which is called a \emph{Berkovich (open) disk}. If $\xi\in\mathsf{A}^1$ and $\infty\not\in\mathsf{B}(\overrightarrow{\xi\xi_1})$, we sometimes write $\mathsf{B}(\overrightarrow{\xi\xi_1})$ by $\mathsf{B}(x,r)$ for any $x\in\mathsf{B}(\overrightarrow{\xi\xi_1})\cap K$ and $r=\mathrm{diam}(\xi)$. Then we in fact have $\mathsf{B}(x,r)\cap K=D(x,r)$. 

The complement $\mathsf{H}^1:=\mathsf{P}^1\setminus\mathbb{P}^1$ is the \emph{(Berkovich) hyperbolic space}. It carries a path distance metric $\rho$ defined as follows. From the tree structure of  $\mathsf{P}^1$, for any two distinct points $\xi_1,\xi_2\in\mathsf{H}^1$, setting $\xi_1\vee\xi_2$ the unique intersection point of the segments $[\xi_1,\infty[$ and $[\xi_2,\infty[$ and  recalling $p$ in \eqref{q}, we define
$$\rho(\xi_1,\xi_2)=2\log_p\mathrm{diam}(\xi_1\vee\xi_2)-\log_p\mathrm{diam}(\xi_1)-\log_p\mathrm{diam}(\xi_2).$$
The metric $\rho$ on $\mathsf{H}^1$ induces the strong topology, which is strictly finer than the weak topology restricted on $\mathsf{H}^1$. We can extend $\rho$ to $\mathsf{P}^1$ by setting $\rho(\xi_1,\xi_2)=\infty$ if at least one of $\xi_1$ and $\xi_2$ is a type I point.

\subsection{Polynomial maps} 
Let $f\in K[z]$ be a nonconstant polynomial. In this subsection. we consider the polynomial endomorphism on $\mathsf{P}^1$ induced by $f$, denoted also by $f$. 

\subsubsection{Basic dynamics}\label{sec:dynamics}
In this subsection, we let $f$ have degree at least $2$. Its \emph{basin} of $\infty$ is defined by 
$$\Omega_\infty(f):=\{\xi\in\mathsf{P}^1:f^n(\xi)\to\infty,\ \text{as}\  n\to\infty\}.$$
Then $\Omega_\infty(f)\not=\emptyset$ since $f(\infty)=\infty\in\Omega_\infty(f)$.
The complement $\mathcal{K}(f):=\mathsf{A}^1\setminus\Omega_\infty(f)$ is called the \emph{filled Julia set} of $f$. Since $f$ has periodic points of type I and such points are not contained in $\Omega_\infty(f)$, the set $\mathcal{K}(f)\cap K$ is nonempty. The boundary $J(f):=\partial\mathcal{K}(f)$ is the \emph{(Berkovich) Julia set} of $f$, and its complement $F(f):=\mathsf{P}^1\setminus J(f)$ is \emph{the (Berkovich) Fatou set} of $f$. For the equivalent definitions, we refer \cite[Section 10.5]{Baker10} and \cite[Section 8.1]{Benedetto19}. We denote by $J_I(f):=J(f)\cap\mathbb{P}^1$ the \emph{classical Julia set} of $f$ and  by $F_I(f):=F(f)\cap\mathbb{P}^1$ the \emph{classical Fatou set} of $f$.

Pick a periodic point $x\in K$ and consider the smallest closed disk $\overline{D}(x,r)$ containing $\mathcal{K}(f)\cap K$. Let $\xi_f\in\mathsf{A}^1$ be the point corresponding to the disk $\overline{D}(x,r)$ and call it the \emph{base point} of $f$. It follows from \cite[Proposition 6.7]{Rivera00} that the point $\xi_f$ is of type II. We say $f$ is \emph{simple} if $J(f)=\{\xi_f\}$; this is the only case that $J(f)$ is a singleton. Otherwise, we say $f$ is \emph{nonsimple}; in this case we have the following.

\begin{lemma}\cite[Proposition 3.4]{Trucco14}\label{lem:base}
If $f\in K[z]$ be a nonsimple polynomial of degree at least $2$, then 
\begin{enumerate}
\item $\{\xi_f\}=f^{-1}(f(\xi_f)$;
\item $\xi_f\prec f(\xi_f)$;
\item $\mathrm{diam}(f^n(\xi_f))\to\infty$, as $n\to +\infty$;
\item $\xi_f$ has at least two preimages under $f$, and any two preimages $\xi_1,\xi_2$ of $\xi_f$, if $\xi_1\preceq\xi_2$, then $\xi_1=\xi_2$; and
\item the preimages of $\xi_f$ under $f$ are contained in at least two distinct directions in $T_{\xi_f}\mathsf{A}^1$. 
\end{enumerate}
\end{lemma}

Denote by $R_f:=\mathrm{diam}(\xi_f)$. By Lemma \ref{lem:base}(1), we have that 
$\deg_{\xi_f}f=d$ and hence
 \begin{equation*}\label{equ:ratio-diam}
\mathrm{diam}(f(\xi_f))=|a_d|R_f^d,
\end{equation*}
 where $a_d\in K\setminus\{0\}$ is the leading coefficient of $f$.
By Lemma \ref{lem:base}, we also have that $J(f)\subset\mathsf{P^1}\setminus\mathsf{B}(\overrightarrow{\xi_f\infty})$.

The periodic cycles of $f$ in $\mathbb{P}^1$ are classified to be \emph{attracting}, \emph{indifferent} or \emph{repelling} according to the absolute value of its multiplier is less than, equal to, or more than $1$. For a periodic cycle in $\mathsf{H}^1$, it is \emph{indifferent} or \emph{repelling} if the local degree of the iteration of $f$ at this cycle  is equal to or more than $1$. All the repelling cycles are in $J(f)$ and the classical nonrepelling cycles are in $F(f)$, see \cite[Theorem 5.14 and Theorem 8.7]{Benedetto19}. Moreover, we say a point in $J(f)$ is \emph{wandering} if it is not eventually periodic under iteration of $f$.

\begin{lemma}\label{lem:repelling}
Let $f\in K[z]$ be a polynomial of degree at least $2$. If $f$ has no wandering points in $J(f)\cap\mathsf{H}^1$, then every point in $J(f)\cap\mathsf{H}^1$ is eventually mapped to a repelling cycle. 
\end{lemma}
\begin{proof}
	Under the assumption, we have that any point in $J(f)\cap\mathsf{H}^1$ is eventually periodic. Note that for any $\xi\in\mathsf{P}^1$ any direction $\vec{v}\in T_{\xi}\mathsf{P}^1$ and any $n\ge 0$, if $\infty\not\in\mathsf{B}(\vec{v})$, we have  $\infty\not\in f^n(\mathsf{B}(\vec{v}))$. Then by \cite[Theorem 2.1]{Kiwi14}, we conclude that $J(f)$ contains no indifferent periodic cycles. Hence the conclusion follows. 
\end{proof}

For a periodic cycle, its \emph{basin} is the set of points in $\mathsf{P}^1$ whose forward orbits converge to the cycle in the weak topology. Then not only the attracting or indifferent cycles in $\mathbb{P}^1$ but also the repelling cycles in $\mathsf{H}^1$ have nontrivial basins. The basin of a repelling cycle in $\mathsf{H}^1$ may contain a \emph{wandering Fatou component} which is a component of $F(f)$ having infinite forward orbit. We say a wandering Fatou component is \emph{strictly wandering}  if it is not in the basin of any periodic cycle.

\begin{lemma}\label{lem:nonwandering}
Let $f\in K[z]$ be a polynomial of degree at least $2$. If $f$ has no wandering points in $J(f)\cap\mathsf{H}^1$, then $f$ has no strictly wandering component in $F(\phi)$.
\end{lemma}
\begin{proof}
	If $f$ has a strictly wandering Fatou componnet, then the boundary of this wandering component gives a wandering point in $J(f)\cap\mathsf{H}^1$. 
\end{proof}

\subsubsection{Mapping property}
Since $f$ is a polynomial, for any two points $\xi_1\prec\xi_2$ in $\mathsf{A}^1$, we have $f(\xi_1)\prec f(\xi_2)$. Moreover, we have the following observation.

\begin{lemma}\label{lem:obs}
Let $f\in K[z]$ be a polynomial of degree at least $1$. Then 
\begin{enumerate}
\item for any two points $\xi_1\prec\xi_2$ in $\mathsf{H}^1$,
 $$||f||_{\xi_1}<||f||_{\xi_2}$$
 \item  for any type II point $\xi\in\mathsf{H}^1$ and any type I point $x_0\prec \xi$, if $0\not\in f(D(x_0,\mathrm{diam}(\xi)))$, then 
 $$||f||_\xi=|f(x_0)|.$$
 \end{enumerate}
\end{lemma}
\begin{proof}
For statement (1), pick $x\in K$ with $x\prec\xi_1$ and set $y:=f(x)$. Then $||f-y||_{\xi_1}=\mathrm{diam}(f(\xi_1))$ and $||f-y||_{\xi_2}=\mathrm{diam}(f(\xi_2))$. Hence $||f-y||_{\xi_1}<||f-y||_{\xi_2}$ since $f(\xi_1)\prec f(\xi_2)$. It follows that $||f||_{\xi_1}<||f||_{\xi_2}$.

For statement (2), since $0\not\in f(D(x_0,\mathrm{diam}(\xi)))$, we have that $||f-f(x_0)||_\xi=\mathrm{diam}(f(\xi))$. If $\mathrm{diam}(f(\xi))<|f(x_0)|$, the conclusion follows immediately. If $\mathrm{diam}(f(\xi))=|f(x_0)|$, then we have $f(\xi)\in]0,\infty[$. Since $f$ maps disk to disk, we conclude that $||f||_\xi=|f(x_0)|$.
\end{proof}

If $f$ has degree $d\ge 2$, its derivative $f'\in K[z]$ is a polynomial of degree $d-1$, which extends to $\mathsf{P}^1$. We will apply Lemma \ref{lem:obs} for $f'$. Moreover, in this case  a $\emph{critical point}$ of $f$ is a point in $\mathbb{P}^1$ at which $f'$ vanishes in local coordinate; then the point $\infty$ is a critical point of $f$, and we denote by $\mathrm{Crit}(f)\subset K$ the set of critical points of $f$ differing from $\infty$.

Let us state the following mapping properties for polynomials. In fact, such properties also hold for more general convergent power series, see \cite[Section 3.4]{Benedetto19}, so here we omit the proof.

\begin{lemma}\label{lem:basic}
	Let $f\in K(z)$ be a polynomial of degree at least $2$. Then the following hold.
	\begin{enumerate}
		\item If $c\in\mathrm{Crit}(f)$, then there exist $0<\varepsilon_c<1$, $m_c\ge 2$ and $C>0$ such that 
		$$\mathrm{diam}(f(D(c,\varepsilon_c)))=C\mathrm{diam}(D(c,\varepsilon_c))^{m_c}.$$
		\item If $f$ is one-to-one on a disk $D\subset K$, then $f(D)\subset K$ is a disk. Moreover, for any $x\in D$
		$$\mathrm{diam}(f(D))=|f'(x)|\mathrm{diam}(D).$$
	\end{enumerate}
\end{lemma}

The $m_c$ in Lemma \ref{lem:basic} is the local degree $\deg_cf$ of $f$ at $c$. In fact, the local degree $\deg.f$ of $f$ at type I points extends to $\mathsf{P}^1$ upper semicontinuously with respect to the weak topology, see \cite[Section 3.1]{Faber13I}. The ramification locus of $f$ is 
$$\mathcal{R}_f:=\{\xi\in\mathsf{P}^1:\deg_\xi f\ge 2\},$$
which is connected and contains the convex hull of $\mathrm{Crit}(f)\cup\{\infty\}$, see \cite[Theorem C]{Faber13I}. Thus for any $\xi_1\in\mathcal{R}_f$ and $\xi_2\in\mathsf{P}^1$, if $\xi_1\prec\xi_2$  we have $2\le\deg_{\xi_1}f\le\deg_{\xi_2}f$; moreover, if in addition $\xi_2$ is sufficiently close to $\xi_1$, we even have $2\le\deg_{\xi_1}f=\deg_{\xi_2}f$.  Recall that  $f$ is tame if $\mathcal{R}_f$ is local finite; in this case,  $\mathcal{R}_f$ is in fact a finite tree.

Now we begin to estimate $||f'||_\xi$.  For $\xi\in\mathsf{H}^1$, define the \emph{distortion} 
$$\delta(f,\xi):=\log_p\mathrm{diam}(\xi)+\log_p||f'||_\xi-\log_p||f||_\xi.$$
We refer \cite[Section 3]{Benedetto14} and \cite[Section 4.1]{Jacobs19} for more details about this distortion. Using this distortion, we have the following estimate on $||f'||_\xi$. Recall that 
\begin{equation}\label{equ:kappa-f}
\kappa:=\kappa(f)=\min\{\log_p|\deg_\xi f|:\xi\in\mathsf{P}^1\}.
\end{equation}

\begin{proposition}\label{prop:distortion}
Let $f\in K[z]$ be a polynomial of degree at least $2$. Then for any $\xi\in\mathrm{H}^1$, 
\begin{equation}\label{equ:diam1}
\log_p||f'||_\xi\ge\kappa+\log_p\mathrm{diam}(f(\xi))-\log_p\mathrm{diam}(\xi);
\end{equation}
in particular, 
$$||f'||_\xi\ge p^{\kappa}\frac{\mathrm{diam}(f(\xi))}{\mathrm{diam}(\xi)}.$$
Moreover,  if $f$ is tame, then 
\begin{equation}\label{equ:diam2}
\log_p||f'||_\xi=\log_p\mathrm{diam}(f(\xi))-\log_p\mathrm{diam}(\xi);
\end{equation}
in particular, 
$$||f'||_\xi=\frac{\mathrm{diam}(f(\xi))}{\mathrm{diam}(\xi)}.$$
\end{proposition}
\begin{proof}
By continuity and density, we assume $\xi$ is of type II.
The proof of \eqref{equ:diam1} follows from the proof of \cite[Lemma 5]{Jacobs19} and the fact that $\deg_\xi f=\#\{f^{-1}(f(x_0))\cap D_{\xi}\}$ where $D_\xi\subset K$ is corresponding closed disk of $\xi$ and $x_0\in D_\xi$ (See \cite[Theorem 3.3]{Benedetto19}). To see \eqref{equ:diam2}, by \cite[Lemma 3.3]{Benedetto14}, we have $\delta(f-f(x_0),\xi)\le 0$, and again applying the proof of \cite[Lemma 5]{Jacobs19}, from the tameness assumption, we in fact conclude that  $\delta(f-f(x_0),\xi)=0$, which implies the desired equality.
\end{proof}

\subsubsection{Further estimate on derivatives}\label{sec:derivative} 
In this subsection, we states an estimate on derivatives of iteration. We let $f$ have degree at least $2$ and continue to use the notations in previous subsections.
First consider the base point $\xi_f$ of $f$ and denote by
\begin{equation}\label{equ:xi}
\Xi_f:=p^{-\kappa}||f'||_{\xi_f}=\max\left\{||f'||_{\xi_f}/|\deg_\xi f|: \xi\in\mathsf{P}^1\right\}.
\end{equation}
Observe that $||f'||_{\xi_f}\le \Xi_f$. Moreover, if $J_I(f)\not=\emptyset$, then for any $x\in J_I(f)$, we have $|f'(x)|\le||f'||_{\xi_f}\le\Xi_f<+\infty$.

\begin{lemma}\label{lem:>1}
Let $f\in K[z]$ be a polynomials of degree at least $2$. Then $\Xi_f\ge1$. Moreover, if $f$ is nonsimple, then $\Xi_f>1$.
\end{lemma}
\begin{proof}
Note that $\xi_f\preceq f(\xi_f)$ and hence $\mathrm{diam}(\xi_f)\le\mathrm{diam}(f(\xi_f))$. By Proposition \ref{prop:distortion}, we immediately have $\Xi_f\ge1$. If $f$ is nonsimple, then $\xi_f\prec f(\xi_f)$. Again by Proposition \ref{prop:distortion}, we have $\Xi_f>1$.
\end{proof}
 
 In our next result (Proposition \ref{lem:kappa}), we will state a lower bound of certain derivatives, which is a nonarchimedean analog of the estimates in \cite[Section 1]{Przytycki93}. Let us first set up the following notations.
Recall that $R_f=\mathrm{diam}(\xi_f)$ and consider the closed disk  $D_{\xi_f}\subset K$ corresponding to $\xi_f$. For $s>0$ and $x\in D_{\xi_f}$, denote by $D_s(x):=D(x,R_f\Xi_f^{-s})\subset K$. It follows from Proposition \ref{prop:distortion} that $f^s(D_s(x))\subset D_{\xi_f}$. Let $\mathsf{B}_s(x):=\mathsf{B}(x,R_f\Xi_f^{-s})\subset \mathsf{A}^1$. Then $\mathsf{B}_s(x)\cap K=D_s(x)$. In the following proofs, if the maps are clear, we will omit the subscripts in the corresponding quantities. 

 We first show the following finiteness result on the critical values of iterations in small neighborhoods of any Julia critical point. 
\begin{lemma}\label{lem:crit-value}
Suppose $f\in K[z]$ be a polynomial of degree $d\ge 2$ and assume $\mathrm{Crit}(f)\cap J(f)\not=\emptyset$. Then there exists a constant $M\ge 1$ such that $f^n$ has at most $M$ critical values in $D_n(c)$ for any $c\in\mathrm{Crit}(f)\cap J(f)$ and any $n\gg1 $.
\end{lemma}
\begin{proof}
First fix a critical point $c'\in\mathrm{Crit}(f)\cap J(f)$. We claim that there exists some constant $C>0$ such that for any $n\gg1$ and $0\le\ell_1<\ell_2\le n$, if $f^{\ell_1}(c')\in D_n(c)$ and $f^{\ell_1}(c')\in D_n(c)$, then $\ell_1-\ell_1\ge Cn$. Form this claim we obtain the conclusion by setting $M=(C^{-1}+1)\cdot (d-1)$. Now let us show the claim.

Observing that $f^{\ell_2-\ell_1}(f^{\ell_1}(c'))\in f^{\ell_2-\ell_1}(D_n(c))$ and noting that $c\in J(f)$, we conclude that 
\begin{equation}\label{equ:include}
	D_n(c)\subset f^{\ell_2-\ell_1}(D_n(c)).
\end{equation}
 Let $\varepsilon_c>0$ and $m_c\ge 2$ be as in Lemma \ref{lem:basic} (1), and observe that $f$ is nonsimple since $J_I(f)\not=\emptyset$. By Lemma \ref{lem:>1}, we have $\Xi>1$, and then considering $n$ sufficiently large such that $\Xi^{-n}<\varepsilon_c$,  we have 
\begin{equation}\label{equ:=}
\mathrm{diam}(f(D_n(c)))=C_1R^{m_c}\Xi^{-m_cn}
\end{equation}
for some $C_1>0$.  Since  $f^\ell(D_n(c))\subset D_{\xi_f}$ for any $1\le\ell\le n$, by \eqref{equ:=} and Proposition \ref{prop:distortion}, we conclude that 
\begin{equation}\label{equ:diam-less}
	\mathrm{diam}(f^{\ell_2-\ell_1}(D_n(c)))=\mathrm{diam}(f^{\ell_2-\ell_1-1}(f(D_n(c))))\le C_1R^{m_c}\Xi^{-m_cn}\Xi^{\ell_2-\ell_1-1}.
\end{equation}
Then by \eqref{equ:include} and \eqref{equ:diam-less}, we obtain
$$R\Xi^{-n}\le  C_1R^{m_c}\Xi^{-m_cn}\Xi^{\ell_2-\ell_1-1}.$$ 
Thus 
$$\ell_2-\ell_1\ge(m_c-1)n+\frac{\log_p(\Xi/(C_1R^{m_c-1}))}{\log_p\Xi}.$$
\end{proof}

Since a component of the primages of a Berkovich open disk in $\mathsf{A}^1$ under $f$ is also a  Berkovich open disk in $\mathsf{A}^1$, together with the previous lemmas, we have the following estimate of derivative. 





\begin{proposition}\label{lem:kappa}
Let $f\in K[z]$ be a polynomial of degree at least $2$. Assume that $c\in\mathrm{Crit}(f)\cap J(f)\not=\emptyset$ and pick $x_0\in K$. Then there exists $\theta>0$ such that for every $n\gg 1$, the following holds: if $f^n(x_0)\in D_{n}(c)$, then there exist $j$ and $\beta$ with $0\le j<j+\beta\le n$ and $\beta\ge\theta n$ such that 
\begin{equation}\label{equ:der}
|(f^\beta)'(f^j(x_0))|\ge\Xi^{-2n}p^{n\kappa}.
\end{equation}
\end{proposition}
\begin{proof}
Consider the component $B^{(n)}$ of $f^{-n}(\mathsf{B}_{n}(c))$ containing $x_0$ and set $D^{(n)}=B^{(n)}\cap\mathbb{P}^1$. Then $D^{(n)}$ is a disk which is disjoint with $\mathsf{B}(\overrightarrow{\xi_f\infty})$. It follows that $\mathrm{diam}(D^{(n)})$ has a uniform upper bound. 
Thus there exists a constant $C>0$ such that 
\begin{equation}\label{equ:ratio-lower}
\frac{\mathrm{diam}(f^n(D^{(n)}))}{\mathrm{diam}(D^{(n)})}=\frac{\mathrm{diam}(\mathsf{B}_{n}(c))}{\mathrm{diam}(D^{(n)})}\ge C\Xi^{-n}.
\end{equation}

Let $M\ge 0$ be as in Lemma \ref{lem:crit-value}. Then there exist $1\le N\le M+1$ and $0\le j_{\ell-1}<j_{\ell}\le n$ for $1\le \ell\le N$ with $j_0=0$ such that $f^{j_\ell}(D^{(n)})$ contains a critical point $c_\ell\in\mathrm{Crit}(f)$. Then we have
\begin{align}\label{equ:ratio}
	\frac{\mathrm{diam}(f^n(D^{(n)}))}{\mathrm{diam}(D^{(n)})}&=\prod_{j=1}^n\frac{\mathrm{diam}(f^j(D^{(n)}))}{\mathrm{diam}(f^{j-1}(D^{(n)})}\nonumber \\
	&=\frac{\mathrm{diam}(f(D^{(n)}))}{\mathrm{diam}(D^{(n)})}\prod_{\ell=0}^{N-1}\frac{\mathrm{diam}(f^{j_{\ell+1}}(D^{(n)}))}{\mathrm{diam}(f^{j_\ell+1}(D^{(n)})}\prod_{\ell=1}^N\frac{\mathrm{diam}(f^{j_{\ell}+1}(D^{(n)}))}{\mathrm{diam}(f^{j_\ell}(D^{(n)}))}\frac{\mathrm{diam}(f^{n}(D^{(n)}))}{\mathrm{diam}(f^{j_N+1}(D^{(n)}))}.
\end{align}

Observe that for sufficiently large $n$, we can obtain that $$\mathrm{diam}(f^{j_\ell}(D^{(n)}))<\varepsilon_{c_\ell},$$
where $\varepsilon_{c_\ell}>0$ is as in Lemma \ref{lem:basic}(1). It follows that $\prod_{\ell=1}^N\frac{\mathrm{diam}(f^{j_{\ell}+1}(D^{(n)}))}{\mathrm{diam}(f^{j_\ell}(D^{(n)}))}$ has a uniform upper bound. Then from \eqref{equ:ratio-lower} and \eqref{equ:ratio}, there exists a constant $C_1>0$ such that 
\begin{equation}\label{equ:ratio-noncrit}
	\prod_{\ell=0}^{N-1}\frac{\mathrm{diam}(f^{j_{\ell+1}}(D^{(n)}))}{\mathrm{diam}(f^{j_\ell+1}(D^{(n)})}\frac{\mathrm{diam}(f^{n}(D^{(n)}))}{\mathrm{diam}(f^{j_N+1}(D^{(n)}))}>C_1\Xi^{-n}.
\end{equation}
For $0\le\ell\le N-1$, set $\beta_\ell=j_{\ell+1}-j_\ell-1$ and set  $\beta_N=n-j_N-1$. Applying Proposition \ref{prop:distortion} to \eqref{equ:ratio-noncrit}, we conclude that 
\begin{equation}\label{equ:der-noncrit}
	p^{-(n-N)\kappa}\prod_{\ell=0}^N|(f^{\beta_\ell})'(f^{j_\ell+1}(x_0))|>C_1\Xi^{-n}.
\end{equation}
It follows that there exists $C_2\ge 0$ such that 
\begin{equation}\label{equ:der-noncrit}
	\prod_{\ell=0}^N|(f^{\beta_\ell})'(f^{j_\ell+1}(x_0))|>C_2p^{n\kappa}\Xi^{-n}.
\end{equation}

Since $n\gg 1$, for any $0<\theta<(n-N)/(nN+n)$, there exists $0\le\ell_0\le N$ such that $\beta_{\ell_0}>\theta n$. We can further take $0<\theta<1/(2N)$.
Now suppose to the contrary that \eqref{equ:der} fails. Since $\Xi>1$ and $\kappa\le 0$, applying  Lemma \ref{lem:obs}(1) to $f'$, we have
\begin{align*}
	\prod_{\ell=0}^N|(f^{\beta_\ell})'(f^{j_\ell+1}(x_0))|&=\prod_{\beta_{\ell}>\theta n}|(f^{\beta_\ell})'(f^{j_\ell+1}(x_0))|\prod_{\beta_{\ell}\le\theta n}|(f^{\beta_\ell})'(f^{j_\ell+1}(x_0))|\\
	&\le \Xi^{-2n}p^{n\kappa}||f'||_{\xi_f}^{\theta n N}\le \Xi^{-2n}p^{n\kappa(1+\theta N)}\Xi^{\theta n N}<p^{n\kappa}\Xi^{-3n/2},
\end{align*}
which contradicts \eqref{equ:der-noncrit}.

\end{proof}

\section{Trees for Polynomials}\label{sec:tree}
In this section, we assume that $f\in K[z]$ is a tame and nonsimple polynomial of degree $d\ge2$. Recall from previous section that $\xi_f\in\mathsf{P}^1$ is the base point of $f$ and $R_f:=\mathrm{diam}(\xi_f)$. We now introduce several trees associated to $f$.

\subsection{Trucco's tree}
In this subsection, we briefly introduce the Trucco's tree for $f$. For more details, we refer \cite{Trucco14}. 

Set $\mathcal{L}_0:=\{\xi_f\}$, and for $n\ge 0$, denote by $\mathcal{L}_n:=f^{-n}(\mathcal{L}_0)$. We say the set $\mathcal{L}_n$ is the level $n$ dynamical set of $f$ and each point in $\mathcal{L}_n$ is a level $n$ dynamical point of $f$. Recalling Lemma \ref{lem:base}, we say that a \emph{dynamical sequence} of $f$ is a decreasing sequence $\{L_n\}_{n\ge 0}$ of dynamical points such that $L_n\in\mathcal{L}_n$. Each point in $J(f)$ can be approached by a dynamical sequence, see \cite[Proposition 3.6]{Trucco14}, that is 
\begin{equation}\label{equ:J}
J(f)=\large\{\lim_{n\to+\infty}L_n : \{L_n\}_{n\ge 0}\ \text{is a dynamical sequence of}\ f\large\}.
\end{equation}
 Lemma \ref{lem:base}(4) implies that if $\{L_n\}_{n\ge 0}$ and $\{L'_n\}_{n\ge 0}$ are two distinct dynamical sequences of $f$, then $\lim\limits_{n\to+\infty}L_n\not=\lim\limits_{n\to+\infty}L'_n$. Thus for each $\xi\in J(f)$, there is a unique dynamical sequences $\{L_n(\xi)\}_{n\ge 0}$ with $L_n(\xi)\in\mathcal{L}_n$ such that $\xi=\lim\limits_{n\to+\infty}L_n(\xi)$. We say $\{L_n(\xi)\}_{n\ge 0}$ is the dynamical sequence for $\xi$.
 
 Lemma \ref{lem:base} and \eqref{equ:J} imply that the closure of the convex hull of $\cup_{n\ge 0}\mathcal{L}_n$ coincides with the convex hull $\mathrm{Hull}(J(f))$ of $J(f)$. Now let $\{c_1,\cdots,c_\ell\}$ be the set of the critical points of $f$ in $\Omega_\infty(f)$.  By \cite[Remark 4.4]{Trucco14}, we have that $1\le \ell\le d-1$ and $c_j\preceq\xi_f$ for each $1\le j\le \ell$.  Denote by $\xi_{c_j}\in\mathsf{A}^1$ the point such that 
$$\mathrm{Hull}(J(f))\cap[c_j,\xi_f]=[\xi_{c_j},\xi_f].$$
Then by \cite[Proposition 4.3]{Trucco14}, the branch points of $\mathrm{Hull}(J(f))$ are contained in the grand orbit $\mathrm{GO}(\{\xi_{c_1},\cdots,\xi_{c_\ell}\})$ of the set $\{\xi_{c_1},\cdots,\xi_{c_\ell}\}$. Moreover, $\xi_{c_j}\in\mathrm{Hull}(J(f))$ for all $1\le j\le \ell$ and $\xi_f\in\{\xi_{c_1},\cdots,\xi_{c_\ell}\}$.

For $\xi\in J(f)$, the \emph{geometric sequence} $\{G_n(\xi)\}_{n\ge 0}$ of $\xi$ is the decreasing sequence of the elements in the intersection $[\xi_f,\xi]\cap \mathrm{GO}(\{\xi_{c_1},\cdots,\xi_{c_\ell}\})$. Then such a sequence is unique for $\xi$. The \emph{generators} of the geometric sequences of  $f$ is the intersection 
$$\{\zeta_0,\cdots,\zeta_{q-1}\}:= \mathrm{GO}(\{\xi_{c_1},\cdots,\xi_{c_\ell}\})\cap\ ]\xi_f,f(\xi_f)]$$ 
with some $1\le q\le \ell$ and $\zeta_j\prec\zeta_{j-1}$ for $1\le j\le q$. Since $\xi_f\in\{\xi_{c_1},\cdots,\xi_{c_\ell}\}$, we in fact have $\zeta_0=f(\xi_f)$. Moreover, any point in $\mathrm{GO}(\{\xi_{c_1},\cdots,\xi_{c_s}\})\cap\mathrm{Hull}(J(f))$ eventually maps to $\zeta_j$ for some $0\le j\le q-1$.

For a point $\xi\in J(f)$, consider its dynamical sequence $\{L_n(\xi)\}_{n\ge 0}$ and geometric sequence $\{G_n(\xi)\}_{n\ge 0}$. We have $L_n(\xi)=G_{qn}(\xi)$ for all $n\ge 0$, and $f(G_n(\xi))=G_{n-q}(f(\xi))$ for all $n\ge q$, see \cite[Proposition 4.9]{Trucco14}.

The \emph{Trucco's tree} $\mathcal{T}_f$ for $f$ is the the convex hull of the grand orbit $\mathrm{GO}(\{\zeta_0,\cdots,\zeta_{q-1}\})$ of the generators with vertices at $\mathrm{GO}(\{\zeta_0,\cdots,\zeta_{q-1}\})$. As a topological space, $\mathcal{T}_f$ can also be written as
$$\mathcal{T}_f=\left(\mathrm{Hull}(J(f))\cup\ ]\xi_f,\infty[\ \right)\setminus J(f).$$
 Observe that $\mathcal{T}_f\subset\Omega_\infty(f)$ and all the vertices of $\mathcal{T}_f$ are type II points. Moreover, for an edge $e$ of $\mathcal{T}_f$, if $e\subset\mathcal{T}_f\setminus\ ]\xi_f, \infty[$, there exists a point $\xi\in J(f)$ with geometric sequence $\{G_n(\xi)\}_{n\ge 0}$ such that $e=]G_{n_0+1}(\xi), G_{n_0}(\xi)[$ for some $n_0\ge 0$; and if $e\subset ]\xi_f, \infty[$, setting $\zeta_q=\xi_f$, there exist $0\le j\le q-1$ and $\ell\ge 0$ such that $e=]f^\ell(\zeta_{j+1}), f^\ell(\zeta_j)[$.

The following result asserts that the local degree is constant on each edge of $\mathcal{T}_f$.

\begin{lemma}\cite[Proposition 4.10]{Trucco14}\label{lem:deg constant}
Let $f\in K[z]$ be a tame and nonsimple polynomial of degree at least $2$ and let $]\xi_1,\xi_2[$ be an edge of the tree $\mathcal{T}_f$ with $\xi_1\prec\xi_2$. Then for any $\xi\in ]\xi_1,\xi_2[$, 
$$\deg_{\xi}f=\deg_{\xi_1}f.$$
\end{lemma}

Moreover, from the upper semicontinuity of the local degree $\deg_{\cdot} f$ on $\mathsf{P}^1$, we immediately have that $\deg_{\xi_1}f\le\deg_{\xi_2}f$ for the edge $]\xi_1,\xi_2[$ in Lemma \ref{lem:deg constant}.

At any type II point $\xi\in\mathsf{H}^1$, the polynomial $f$ induces a natural map $T_\xi f: T_\xi\mathsf{P}^1\to T_{f(\xi)}\mathsf{P}^1$. For a direction $\vec{v}\in T_\xi\mathsf{P}^1$, if $\vec{v}\not=\overrightarrow{\xi\infty}$, then $f(\mathsf{B}(\vec{v}))=\mathsf{B}(T_\xi f(\vec{v}))$, see \cite[Theorem 9.42]{Baker10}. It follows that $f$ maps an edge of $\mathcal{T}_f$ onto an edge of $\mathcal{T}_f$ homeomorphically. Then by Lemma \ref{lem:deg constant} and \cite[Lemma 9.33]{Baker10}, for an edge $e:=]\xi_1,\xi_2[$ of $\mathcal{T}_f$ with $\xi_1\prec\xi_2$, the lengths
\begin{equation*}
	\rho(f(e))=\rho(f(\xi_1), f(\xi_2))=\deg_{\xi_1}f\cdot\rho(\xi_1,\xi_2)=\deg_{\xi_1}f\cdot\rho(e),
\end{equation*}
which we will apply repeatedly in the following sections to calculate the lengths for certain segments.

\subsection{Marked grid}\label{sec:grid}
The marked grid has been introduced by Branner and Hubbard \cite[Section 4]{Branner92} for complex cubic polynomials with one escaping critical point, also see \cite[Section 2]{Harris99}. In the nonarchimedean setting, the marked grids has been studied by Kiwi \cite[Section 4.1]{Kiwi06} for cubic polynomials. In this subsection, we further assume that $\mathrm{Crit}(f)\cap J(f)\not=\emptyset$ and consider the marked grids of $f$ with respect to a marked critical point $c_0\in\mathrm{Crit}(f)\cap J(f)$.


For $\xi\in\mathcal{K}(f)$, there exists a unique point $\xi'\in[\xi,\infty[\ \cap J(f)$. Consider the geometric sequence $\{G_n(\xi')\}_{n\ge 0}$ of $\xi'$ and set $G_n(\xi):=G_n(\xi')$. For $j\ge 1$, denote by $e_j(\xi):=]G_{j-1}(\xi),G_j(\xi)[\subset\mathcal{T}_f$ the level $j$ edge in $\mathcal{T}_f\setminus\ ]\xi_f,\infty[$ for $\xi$. Moreover, for connivence, denote by $e_{0}:=]\zeta_{q-1},\xi_f[$. For any $j\ge 0$ and $k\ge 0$, set $e_{j,k}(\xi):=e_{j}(f^k(\xi))$. By Lemma \ref{lem:deg constant}, the polynomial $f$ has constant local degree on each $e_{j,k}(\xi)$.

\begin{definition}
For a point $\xi\in \mathcal{K}(f)$, the \emph{marked grid} of $\xi$ is the two dimensional array $\mathbf{M}(\xi)=(M_{j,k}(\xi))$ with $j\ge 0$ and $ k\ge 0$, where 

$$M_{\ell,k}(\xi)=
\begin{cases}
1\ \ \text{if}\ e_{\ell,k}(\xi)=e_\ell(c_0),\\
0\ \ \text{if otherwise}.
\end{cases}
$$

\end{definition}

If  $M_{\ell,k}(\xi)=1$, we say the position $M_{\ell,k}(\xi)$ is \emph{marked}. Since $e_0(\xi)=]\zeta_{q-1},\xi_f[=e_0(c_0)$, the position $M_{0,k}(\xi)$ is marked for all $k\ge 0$. Moreover, if $\xi=c_0$, then we say $\mathbf{M}(c_0)$ is the \emph{critical marked grid}; in this case $M_{\ell,0}(c_0)$ is marked for all $\ell\ge 0$.



For the critical marked grid $\mathbf{M}(c_0)$, fixing $k\ge1 $ and $m\ge k$, for any $0\le j\le m+1-k$, we set $\gamma_{m,k}(j)\ge 1$ to be the smallest positive integer such that $M_{j+q\gamma_{m,k}(j),k-\gamma_{m,k}(j)}(c_0)$ is marked, where $q$ is the number of the generators of the geometric sequences of $f$. To see the existence of such $\gamma_{m,k}(j)$, we start at $M_{j,k}(c_0)$ in $\mathbf{M}(c_0)$ and move towards the southwest by moving $qs$ rows when moving $s$ columns until hitting a marked position, then the number of columns moved to the left is $\gamma_{m,k}(j)$. Observe that  $1\le \gamma_{m,k}(j)\le k$. Moreover, $\gamma_{m,k}(j)$ is nondecreasing and $j+ q\gamma_{m,k}(j)$ is strictly increasing on $j$.

\subsection{Omitted tree}

In this subsection, we introduce another tree associated to $f$, called omitted tree, which, roughly speaking, omits a critical point.

Fix a critical point $c\in\mathrm{Crit}(f)$. Let $\mathcal{S}_c\subset\mathsf{P}^1$ be the convex hull of $\left(\mathrm{Crit}(f)\setminus\{c\}\right)\bigcup\{\infty\}$. Then $\mathcal{S}_c\subset\mathcal{R}_f$.
Define
$$H_c:=\bigcup_{j\ge 0}f^j(\mathcal{S}_c).$$
Since  $f$ maps $[\xi,\infty]$ onto $[f(\xi),\infty]$ homeomorphically, for any $\xi\in\mathsf{A}^1$,  the tree $H_c$ is the convex hull of $\bigcup_{j\ge 0}f^j\left(\mathrm{Crit}(f)\setminus\{c\}\right)\bigcup\{\infty\}$.

The \emph{omitted tree} $\mathcal{H}_c$ of $f$ with respect to $c$ is the intersection
$$\mathcal{H}_c:=H_c\cap\mathcal{T}_f.$$
Note that $\mathcal{H}_c=\emptyset$ if and only if $f$ is unicritical, that is $\mathrm{Crit}(f)=\{c\}$. Indeed, if $f$ is unicritical, then $H_c=\{\infty\}$, and hence $\mathcal{H}_c=\emptyset$ since $\infty\not\in\mathcal{T}_f$; if $f$ is not unicritcial, then $\xi_f$ is in $H_c$ and hence in $\mathcal{H}_c$.

\begin{lemma}\label{lem:forward}
The tree $\mathcal{H}_c$ is forward invariant. 
\end{lemma}
\begin{proof}
We assume that $\mathcal{H}_c\not=\emptyset$ and pick $\xi\in\mathcal{H}_c$. Then there exists $c'\in\mathrm{Crit}(f)\setminus\{c\}$ and $j\ge 0$ such that $\xi\in ]f^{j}(c'),\infty[$. It follows that $f(\xi)\in ]f^{j_c+1}(c),\infty[\subset H_c$. Moreover, since $\mathcal{T}_f$ is forward invariant, we conclude that $f(\xi)\in \mathcal{H}_c$.
\end{proof}

\begin{corollary}\label{coro:backward}
The complement $\mathsf{A}^1\setminus \mathcal{H}_c$ is backward invariant. 
\end{corollary}

If $f$ has a critical point $c$ in $J(f)$, then $f$ is not unicritical and hence $\mathcal{H}_c$ is nonempty. If in addition $c$ is the unique critical point in $J(f)$ and $f$ has no wandering Julia points in $\mathrm{H}^1$, then $\mathcal{H}_c\setminus\ ]\xi_f,\infty[$ is bounded away from the type I points:



\begin{proposition}\label{lem:bound}
Let $f\in K[z]$ be a tame and nonsimple polynomial of degree at least $2$. Suppose that $J(f)\cap\mathrm{Crit}(f)=\{c\}$ and all points in $J(f)\cap\mathsf{H}^1$ are nonwandering. Then there exists $\alpha:=\alpha(f)>0$ such that for any $\xi\in\mathcal{H}_c\setminus\ ]\xi_f,\infty[$, 
$$0\le \rho(\xi_f,\xi)\le\alpha.$$
\end{proposition}
\begin{proof}
Pick any $\xi\in\mathcal{H}_c\setminus\ ]\xi_f,\infty[$. Let $\eta\in \mathcal{H}_c\setminus\ ]\xi_f,\infty[$ be an endpoint of $\mathcal{H}_c\setminus\ ]\xi_f,\infty[$ such that $\eta\not=\xi_f$ and $\xi\in [\eta, \xi_f]$. Note that there exists $c'\in\mathrm{Crit}(f)\setminus\{c\}$ and $j\ge 0$ such that $[\eta, \xi_f]\subset[f^{j}(c'), \xi_f]$. Since $c$ is the unique critical point of $f$ in $J(f)$, we have $c'\in F(f)$. Then denote by $\Omega\subset\mathsf{P}^1$ the Fatou component of $f$ containing $c'$.

If $\Omega\not=\Omega_\infty(f)$, then $c\in\mathcal{K}(f)\setminus J(f)$. By \eqref{equ:J}, the boundary $\partial\Omega$ is a singleton, denoted by $\{\xi_{c'}\}$. By Proposition \ref{lem:repelling}, the point $\xi_{c'}$ is eventually mapped to a type II  repelling cycle.  Set 
$$\alpha_{c'}:=\max\{\rho(\xi_f,f^\ell(\xi_{c'})): \ell\ge 0\}.$$
It follows that $0<\alpha_{c'}<\infty$ and $0<\rho(\xi_f,\eta)\le\rho(\xi_f,f^j(\xi_{c'}))\le\alpha_{c'}$.

If $\Omega=\Omega_\infty(f)$, let $\xi_{c'}\in\mathcal{T}_f$ be the point such that $]c',\infty[ \ \cap\mathcal{T}_f=[\xi_{c'},\infty[$. It follows that $f^j(\xi_{c'})\preceq\eta$. Consider a point $\xi'\prec \xi_{c'}$ in $J(f)$. Then there exists $1\le j_0<\infty$ such that $\xi_{c'}\in [L_{j_0}(\xi'), L_{j_0-1}(\xi')[$, where $\{L_n(\xi')\}_{n\ge 0}$ is the dynamical sequence for $\xi'$. It follows that $f^{j_0}(\xi_{c'})\in[\xi_f,f(\xi_f)[$ and hence $0\le j<j_0$. Set 
$$\alpha_{c'}:=\max\{\rho(\xi_f,f^\ell(\xi_c)): 0\le\ell\le j_0\}.$$
We conclude that $0<\alpha_{c'}<\infty$ and  $0<\rho(\xi_f,\eta)\le \rho(\xi_f,f^j(\xi_{c'}))<\alpha_{c'}$.

Now set 
$$\alpha:=\max\{\alpha_{c'}: c'\in\mathrm{Crit}(f)\setminus\{c\}\}.$$ 
Then we have $0<\rho(\xi_f,\eta)\le\alpha$. It follows immediately that $0\le \rho(\xi_f,\xi)\le\alpha$.
\end{proof}

\begin{corollary}
Let $f\in K[z]$ be a polynomial as in Proposition \ref{lem:bound}. Then the closure of $\mathcal{H}_c$ is contained in $\mathsf{H}^1$. In particular, any endpoint of $\mathcal{H}_c$ in $\mathsf{A}^1$ is either a vertex of $\mathcal{T}_f$ or a point in $J(f)\cap\mathsf{H}^1$.
\end{corollary}

\section{Lyapunov exponents}\label{sec:Lya}

In this section, we state an approximation of the (lower) Lyapunov exponents via the diameters of geometric sequences introduced in previous section. We let $f\in K[z]$ be a polynomial of degree at least $2$ and let $\mu$ be an $f$-invariant Radon probability measure on the Julia set $J(f)$. Recall \eqref{L-} and \eqref{L}.

For any $z\in K$, the chordal derivative of $f$ is 
\begin{equation}\label{equ:der-0}
f^\#(z)=\frac{\max(1,|z|)^2}{\max(1,|f(z)|)^2}\cdot |f'(z)|.
\end{equation}
It extends continuously to $\mathsf{A}^1$ as
\begin{equation}\label{equ:np}
f^\#(\xi)=\frac{\max(1,||z||_\xi)^2}{\max(1,||f||_\xi)^2}\cdot ||f'||_\xi.
\end{equation}
The chordal derivative satisfies the chain rule, see \cite[Lemma 1]{Jacobs19}; in particular, 
$$(f^n)^\#(\xi)=\prod_{j=0}^{n-1}f^\#(f^j(\xi)).$$
Moreover, observe that $\log_p f^\#$ is $\mu$-integrable if and only if $\log_p ||f'||$ is $\mu$-integrable, since $J(f)\subset\mathsf{P}^1\setminus\mathsf{B}(\overrightarrow{\xi_f\infty})$.
For more details, we refer \cite[Section 2.1]{Jacobs19}. 

\begin{lemma}
Let $f\in K[z]$ be a polynomial of degree at least $2$ and let $\mu$ be an $f$-invariant measure. If $\log_p f^\#$ is $\mu$-integrable, then
\begin{equation}\label{equ:L=}
L(f,\mu)=\int_{\mathsf{P}^1}\log_p||f'||d\mu.
\end{equation}
\end{lemma}
\begin{proof}
For any affine map $A\in K[z]$, the measure $(A^{-1})_\ast\mu$ is an $(A^{-1}\circ f\circ A)$-invariant measure on $J(A^{-1}\circ f\circ A)$. Moreover, direct computations shows 
   \begin{equation}\label{equ:L-conj}
   L(f,\mu)=L(A^{-1}\circ f\circ A,  (A^{-1})_\ast\mu)
   \end{equation}
   (see the proof of  \cite[Lemma 2]{Jacobs19}).  Now recall the base point $\xi_f$ of $f$ and consider the affine map $A_0\in K[z]$ such that $A_0^{-1}(\xi_f)$ is the Gauss point $\xi_G$. Set $g:=A_0^{-1}\circ f\circ A_0$ is $\xi_G$.  Then $\xi_g=\xi_G$. It follows that for any $\xi\in J(g)$, we have $\xi\preceq\xi_G$ and $g(\xi)\preceq\xi_G$. It follows that $||z||_\xi\le1$ and $||g||_\xi\le1$ for any $\xi\in J(g)$. Thus from \eqref{equ:np},
   $$g^\#(\xi)=\frac{\max(1,||z||_\xi)^2}{\max(1,||g||_\xi)^2}\cdot ||g'||_\xi=||g'||_\xi,$$ 
   and hence 
   \begin{equation}\label{equ:L-conj-1}
   L(g, (A_0)^{-1}_\ast\mu)=\int_{\mathsf{P}^1}\log||g'||d (A_0)^{-1}_\ast\mu.
    \end{equation}
   Note that $||g'||=||f'\circ A||$. From \eqref{equ:L-conj} and \eqref{equ:L-conj-1}, we obtain \eqref{equ:L=}.
\end{proof}

Denote by $\xi_0:=0\vee\xi_f$ and let $r=\mathrm{diam}(\xi_0
)$. It follows that $J(f)\subset\overline{D}(0,r)$. Then for any $\xi\in J(f)$, we have 
\begin{equation}\label{equ:z}
1\le \max(1,||z||_\xi)^2\le\max\{1,r^2\}.
\end{equation}
Moreover, by the invariance of $J(f)$, we have that $f^n(\xi)\in J(f)$ for all $n\ge 1$. Then 
\begin{equation}\label{equ:nz}
1\le \max(1,||f^n||_\xi)^2\le\max\{1,r^2\}.
\end{equation}
From \eqref{equ:z} and \eqref{equ:nz}, we conclude that 
\begin{equation}\label{equ:ration-b}
\frac{1}{\max\{1,r^2\}}\le\frac{\max(1,||z||_\xi)^2}{\max(1,||f^n||_\xi)^2}\le\max\{1,r^2\}.
\end{equation}
Noting that from \eqref{equ:np}, for all $n\ge 1$, 
\begin{equation*}
(f^n)^\#(\xi)=\frac{\max(1,||z||_\xi)^2}{\max(1,||f^n||_\xi)^2}\cdot ||(f^n)'||_\xi
\end{equation*}
and applying \eqref{equ:ration-b},
we have that  for any $\xi\in J(f)$, 
\begin{align*}
\liminf_{n\to\infty}\frac{1}{n}\log_p\left(\frac{1}{\max\{1,r^2\}}\cdot ||(f^n)'||_\xi\right)&\le L_f^-(\xi)=\liminf_{n\to\infty}\frac{1}{n}\log_p||(f^n)^\#||_\xi\\
&\le\liminf_{n\to\infty}\frac{1}{n}\log_p\left(\max\{1,r^2\}\cdot ||(f^n)'||_\xi\right);
\end{align*}
which implies that
\begin{equation}\label{equ:L-}
L_f^-(\xi)=\liminf_{n\to\infty}\frac{1}{n}\log_p ||(f^n)'||_\xi.
\end{equation}
Similar argument shows that if $L_f(\xi)$ exists, then
\begin{equation}\label{equ:L}
L_f(\xi)=\lim_{n\to\infty}\frac{1}{n}\log_p ||(f^n)'||_\xi.
\end{equation} 

Thus to bound $L_f^-(\xi)$ and $L_f(\xi)$, we only need to figure out  $||(f^n)'||_\xi$, which, for type I points, is the aim of next lemma. If in addition $f$ is tame, for a point $x\in J(f)$, recall that $\{G_\ell(x)\}_{\ell\ge 0}$ is the geometrical sequence of $x$ and $q\ge 1$ is the number of the generators of the geometric sequences of $f$.



\begin{proposition}\label{lem:ratio}
Let $f\in K[z]$ be a tame and nonsimple polynomial of degree at least $2$. For $z\in J_I(f)$, if $f^\ell(z)\not\in\mathrm{Crit}(f)$ for all $\ell\ge 0$, then for any $n\ge 1$, there exists $N:=N(n)>qn$ such that for any $m\ge N$, 
$$|(f^n)'(z)|=\frac{\mathrm{diam}(G_{m-qn}(f^n(z)))}{\mathrm{diam}(G_m(z))}.$$
\end{proposition}
\begin{proof}
Consider the ramification locus $\mathcal{R}:=\mathcal{R}_f$ of $f$, and for any fixed $n\ge 1$, set 
$$\mathcal{R}_n:=\bigcup_{0\le j\le n-1}f^{-j}(\mathcal{R}).$$
Note that $z$ is not contained in $\mathcal{R}_n$; for, otherwise, we would have $f^j(z)\in\mathrm{Crit}(f)$ for some $j\ge 0$. 
  Then there exists $N>qn$ such that 
$]z,G_{N-1}(z)[\ \cap\mathcal{R}_n=\emptyset$.
It follows that $\deg_{G_m(z)}f^n=1$ for any $m\ge N$. Hence, on the closed disk $\overline{D}(z,r_m)$ corresponding to the point $G_m(z)$, $f$ is one-to-one. By Lemma \ref{lem:basic}(2), we conclude that 
$$|(f^n)'(z)|=\frac{\mathrm{diam}(f^n(\overline{D}(z,r_m)))}{\mathrm{diam}(\overline{D}(z,r_m))}.$$
Observe that $f^n(\overline{D}(z,r_m)$ is the closed disk corresponding to the point $f^n(G_m(z))$. The conclusion follows by noting that $f^n(G_m(z))=G_{m-qn}(f^n(z))$.
\end{proof}

Recall that $\rho$ is the metric on $\mathsf{H}^1$.  Proposition \ref{lem:ratio} implies the following expression of the lower Lyapunov exponent.
\begin{corollary}\label{coro:distance}
Under the assumptions in Proposition \ref{lem:ratio}, for $n\ge 1$, there exists $N:=N(n)>qn$ such that for any $m\ge N$, 
\begin{equation}\label{equ:derivative}
\log_p|(f^n)'(z)|=\rho(\xi_f,G_m(z))-\rho(\xi_f,G_{m-qn}(f^n(z))).
\end{equation}
In particular,
\begin{equation}\label{equ:L-difference}
L_f^-(z)=\liminf_{n\to\infty}\frac{1}{n}\left(\rho(\xi_f,G_m(z))-\rho(\xi_f,G_{m-qn}(f^n(z)))\right).
\end{equation}
\end{corollary}
\begin{proof}
Let $N$ be as in Proposition \ref{lem:ratio}. Note that for any $m\ge N$, we have $G_m(z)\prec\xi_f$ and $G_{m-nq}(z)\prec\xi_f$. It follows that
$$\rho(\xi_f,G_m(z))=\log_p\frac{\mathrm{diam}(\xi_f)}{\mathrm{diam}(G_m(z))},$$
and 
$$\rho(\xi_f,G_{m-qn}(z))=\log_p\frac{\mathrm{diam}(\xi_f)}{\mathrm{diam}(G_{m-qn}(f^n(z)))}.$$
By Proposition \ref{lem:ratio}, we obtain \eqref{equ:derivative}. 
Then applying  \eqref{equ:L-}, we also obtain \eqref{equ:L-difference}.
\end{proof}

\section{Lower bound of Lyapunov exponents}\label{sec:proof1}

As mentioned in Section \ref{sec:intro}, our proof of Theorem \ref{thm:bound} follows the ideas in Przytycki's proof of \cite[Theorem A]{Przytycki93} in complex setting and in Jacobs' proof of \cite[Theorem 2]{Jacobs19}. In this section, we include these proofs.  Recall that $R:=R_f=\mathrm{diam}(\xi_f)$, $\Xi:=\Xi_f$ as in \eqref{equ:xi}, $\kappa$ as in \eqref{equ:kappa-f} and $D_s(x):=D(x,R\Xi^{-s})\subset K$ for $s>0$ and $x\in D_{\xi_f}$. 

\begin{proof}[Proof of Theorem \ref{thm:bound}]

Let us first show statement $(1)$. Let $E$ be the set of $\xi\in J(f)$ satisfying that $L_f(\xi)$ exists or $L_f(\xi)=-\infty$. The Birkhoff Ergodic Theorem implies that $\mu(E)=1$. Now recall \eqref{equ:L} and fix an arbitrary $\sigma>0$. For every integer  $T\ge 1$, denote by $E(T)$ the subset of $\xi\in E$ satisfying that for every $n\ge 1$, 
\begin{equation}\label{equ:m}
\begin{cases} 
\log_p||(f^n)'||_\xi<n(L_f(\xi)+\sigma)+T \ \ &\text{if}\ \ L_f(\xi)>-\infty\\
\log_p||(f^n)'||_\xi<-n\sigma^{-1}+T  \ \ &\text{if}\ \ L_f(\xi)=-\infty. 
\end{cases}
\end{equation}
Then $\cup_{T\ge 1}E(T)=E$. Since $\mu(J_I(f))>0$, there exists $T_0\ge 1$ such that 
\begin{equation}\label{equ:T_0}
\mu(E(T_0)\cap J_I(f))>0.
\end{equation}

Let $g_{T_0}$ be the function on $\mathsf{P}^1$ such that $g_{T_0}=1$ on $E(T_0)$ and $g_{T_0}=0$ on $\mathsf{P}^1\setminus E(T_0)$. Applying the Birkhoff Ergodic Theorem, we obtain that for $\mu$-almost all $\xi\in\mathsf{P}^1$, the limit $h_{T_0}(\xi):=\lim_{n\to\infty}\frac{1}{n}\sum_{\ell=1}^{n-1}(g_{T_0}(f^\ell(\xi)))$ exists and is nonzero, and hence for $\mu$-almost every point in $E(T_0)\cap J_I(f)$, the corresponding limit exists.

For any $T_0$ such that \eqref{equ:T_0} holds, consider $g_{T_0}$ and pick $x\in E(T_0)\cap J_I(f)$ such that $h_{T_0}(x)$ exists.  As  $T_0$ varies, we conclude such $x$ forms a set $X$ with $\mu(X)=\mu(J_I(f))$. Picking any $x\in X$, we show  $L_f(x)$ is bounded below.  Without loss of generality, we assume $x\in E(T_0)\cap J_I(f)$ for some $T_0$.

Considering the distance $S_n$ of $f^n(x)$ and $\mathrm{Crit}(f)$, we have the following two subcases.

Case 1.1: For every $n\gg1$ and $c\in\mathrm{Crit}(f)$, there exists $s>0$ such that $f^n(x)\not\in D_{sn}(c)$, that is, $S_n>R\Xi^{-sn}$. We claim that 
\begin{equation}\label{equ:s}
L_f(x)\ge-s\log_p\Xi+\kappa.
\end{equation}
Suppose to the contrary that $L_f(x)<-s\log_p+\kappa$. It follows that 
$$\limsup_{n\to\infty}\left(\frac{\log_pp^{-n\kappa}}{n}+L_f(x)-\frac{\log_pS_n}{n}\right)<-\kappa-s\log_p\Xi+\kappa+s\log_p\Xi=0.$$
Then there exists $\tau>0$ such that for every $n\ge 0$,
$$\frac{\tau p^{-n\kappa}|(f^n)'(x)|}{S_n}<1.$$
Applying Proposition \ref{prop:distortion}, we conclude that 
$$f^n(D(x,\tau))\subseteq D(f^n(x),\tau p^{-n\kappa}|(f^n)'(x)|).$$ 
It follows that $f^n(\mathsf{B}(x,\tau))\subseteq\mathsf{B}(f^n(x),\tau p^{-n\kappa}|(f^n)'(x)|)$. Hence $\cup_{n=0}^\infty f^n(\mathsf{B}(x,\tau))$ contains no critical points of $f$. Since $x\in J(f)$, by \cite[Theorem 8.15]{Benedetto19} we conclude $\mathrm{Crit}(f)$ is contained in the exceptional set of $f$. Then by \cite[Theorem 1.19]{Benedetto19}, the polynomial $f$ is conjugate to the degree $d$ monomial and hence is simple, which contradicts the assumption $\mu(J_I(f))>0$.

Case 1.2: There exist $n\gg1$ and $c\in\mathrm{Crit}(f)$ such that for any $s>0$, $f^n(x)\in D_{sn}(c)$. Let us take $s=1$ and apply Lemma \ref{lem:kappa}. Since $x\in J_I(\phi)$, so is the critical point $c$. Let $\theta>0$ be as in Lemma \ref{lem:kappa}. 
For $\ell\gg 1$, letting $n_\ell<n_{\ell+1}$ be the consecutive integers, depending on $x$, such that $f^{n_\ell}(x)\in E(T_0)\cap J_I(f)$, and considering $h_{T_0}(x)$, we conclude that 
\begin{equation}\label{equ:number}
n_{\ell+1}-n_\ell<\frac{\theta n_\ell}{2}.
\end{equation}
Moreover, letting  $j\ge 0$ and $\beta\ge\theta n$ be as in Lemma \ref{lem:kappa}, we have 
\begin{equation}\label{equ:bb}
|(f^\beta)'(f^j(x))|\ge \Xi^{-2n}p^{n\kappa}.
\end{equation}
 Note that $x\in E(T_0)$. Since $n\gg1$, it follows from \eqref{equ:number} that there exists $0\le t<\theta n/2$ such that $f^{j+t}(x)\in E(T_0)$. From \eqref{equ:bb}, observing that 
$$\Xi^{-2n}p^{n\kappa}\le |(f^\beta)'(f^j(x))|\le|(f^{\beta-t})'(f^{j+t}(x))|\cdot ||f'||_{\xi_f}^t\le|(f^{\beta-t})'(f^{j+t}(x))|\Xi^t,$$
we have 
$$|(f^{\beta-t})'(f^{j+t}(x))|\ge\Xi^{-2n-\theta n/2}p^{n\kappa}.$$
Then from \eqref{equ:m}, either 
\begin{equation}\label{equ:1}
(\beta-t)(L_f(f^{j+t}(x))+\sigma)+T_0\ge(-2n-\frac{\theta n}{2})\log_p\Xi+n\kappa,
\end{equation}
or 
\begin{equation}\label{equ:1'}
-(\beta-t)\sigma^{-1}+T_0\ge(-2n-\frac{\theta n}{2})\log_p\Xi+n\kappa.
\end{equation}

If $L_f(x)+\sigma\ge 0$ for some $\sigma$, then 
 \begin{equation}\label{equ:0}
 L_f(x)\ge -\sigma,
 \end{equation}
which immediately gives a lower bound of $L_f(x)$. Now assume that $L_f(x)+\sigma<0$ for arbitrary $\sigma$. Note that $L_f(f^{j+t}(x))=L_f(x)$. 
Since  $\beta-t>\theta n/2$, from \eqref{equ:1},  we have
 \begin{equation}\label{equ:1+}
\frac{\theta n}{2}(L_f(x)+\sigma)+T_0\ge(-2n-\frac{\theta n}{2})\log_p\Xi+n\kappa,
\end{equation}
or from \eqref{equ:1'}, we have 
\begin{equation}\label{equ:1'+}
-\frac{\theta n}{2}\sigma^{-1}+T_0\ge(-2n-\frac{\theta n}{2})\log_p\Xi+n\kappa.
\end{equation}
Since $\kappa\le 0$ and $\sigma$ is arbitrary, we conclude that the inequality \eqref{equ:1'+} can not hold. 
So we are in \eqref{equ:1+}. It follows that 
\begin{equation}\label{equ:1++}
L_f(x)\ge(-\frac{4}{\theta}-1)\log_q\Xi+\frac{2\kappa}{\theta}.
\end{equation}

The inequalities \eqref{equ:s}, \eqref{equ:0} and \eqref{equ:1++} show $L_f(x)$ has a lower bound. Now we show that up to a measure zero set, $L_f(x)\ge\kappa$. If $\mathrm{Crit}(f)\cap J(f)=\emptyset$, then we are in case 1.1 for any arbitrary $s>0$. It follows from \eqref{equ:s} that $L_f(x)\ge\kappa$. Now let us assume $\mathrm{Crit}(f)\cap J(f)\not=\emptyset$. From the previous argument and the invariance of $\mu$, we have $\int_{J_I(f)}\log |f'|d\mu$ equals to $\int_{J_I(f)}L_f(x)d\mu$ and hence has a lower bound. The constant $C(f):=||f'||_{\xi_f}$ give an upper bound for $L_f(x)$ and hence for $\int_{J_I(f)}\log |f'|d\mu$. Thus $\log |f'|$ is $\mu$-integrable on $J_I(f)$. Then for any $c\in\mathrm{Crit}(f)\cap J(f)$ and any $s>0$, considering the integration of  $\log |f'|$ on $\mathsf{B}_{sn}(c)$, we have that for $n_0\gg 1$, the exists a constant $C>0$ such that 
\begin{align*}
-\infty&<\int_{\mathsf{B}_{sn_0}(c)}\log_p |f'|d\mu\le C\sum_{n\ge n_0}\int_{\mathsf{B}_{sn}(c)\setminus\mathsf{B}_{s(n+1)}(c)}\log_p|z-c|^{\deg_cf-1}d\mu \\
&\le C(\deg_cf-1)\sum_{n\ge n_0}\int_{\mathsf{B}_{sn}(c)\setminus\mathsf{B}_{s(n+1)}(c)}\log_p (R_f\Xi^{-sn})d\mu \\
&=C(\deg_cf-1)\mu(\mathsf{B}_{n_0})\log_pR-C(\deg_cf-1)s\log_p\Xi\sum_{n\ge n_0}n(\mu(\mathsf{B}_{sn}(c))-\mu(\mathsf{B}_{s{n+1}}(c)))\\
&=C(\deg_cf-1)\mu(\mathsf{B}_{n_0})\log_pR-C(\deg_cf-1)s\left(n_0\mu(\mathsf{B}_{sn_0}(c))+\sum_{n>n_0}\mu(\mathsf{B}_{sn}(c))\right)\log_q\Xi.
\end{align*}
It follows that $\sum_{n>n_0}\mu(\mathsf{B}_{sn}(c))$ is convergent. Hence $\sum_{n>n_0}\mu(f^{-n}(\mathsf{B}_{sn}(c)))$ is convergent since $\mu$ is $f$-invariant. We conclude that $\mu(f^{-n}(\mathsf{B}_{sn}(c)))\to 0$. Thus for $\mu$-almost $\xi\in\mathsf{P}^1$, $f^n(\xi)\not\in\mathsf{B}_{sn}(c))$. Since $\mu(J_I(f))>0$, for $\mu$-almost $y\in J_I(f)$, we have $f^n(y)\not\in\mathsf{B}_{sn}(c)$. Since $s>0$ is arbitrary, by the argument in the case 1.1, we obtain that $L_f(x)\ge\kappa$ for $\mu$-almost all points $x\in E(T_0)\cap J_I(f)$.  This completes the proof of statement $(1$).

\smallskip
For statement $(2)$, note that $||f'||_\xi\le ||f'||_{\xi_f}$ for any $\xi\preceq\xi_f$  by Lemma \ref{lem:obs}. Since $J(f)\cap\mathsf{H}^1\subset\mathsf{P}^1\setminus\mathsf{B}(\overrightarrow{\xi_f\infty})$, we have 
$$\int_{J(f)\cap\mathsf{H}^1}\log_p||f'||_\xi d\mu\le(\log_p||f'||_{\xi_f})\mu(J_I(f)\cap\mathsf{H}^1).$$

Since $\log_p\mathrm{diam}(\cdot)$ is $\mu$-integrable on $J(f)\cap\mathsf{H}^1$, applying Proposition \ref{prop:distortion} and the invariance of $\mu$, we conclude that 
\begin{align*}\label{equ:L=0}
\int_{J(f)\cap\mathsf{H}^1}\log_p||f'|| d\mu&\ge\kappa+\int_{J(f)\cap\mathsf{H}^1}\left(\log_p\mathrm{diam}(f(\xi))-\log_p\mathrm{diam}(\xi)\right)d\mu\\
&=\kappa.
\end{align*}
Hence by \eqref{equ:L=} and the above inequalities, the statement (2) follows.

\smallskip

Now let us assume $\mu$ is ergodic. If $\mu(J_I(f))=1$, we have $L(f,\mu)=\int_{J_I(f)}\log_p|f'|d\mu\ge\kappa$ by statement $(1)$ and the Birkhoff Ergodic Theorem. If $\log_p\mathrm{diam}(\cdot)$ is $\mu$-integrable on $J(f)\cap\mathsf{H}^1$, by statements $(1)$ and $(2)$, we in fact have that $\log_p||f^\#||$ is $\mu$-integral, and then again by the Birkhoff Ergodic Theorem, we obtain the conclusion. 
\end{proof}

\begin{proof}[Proof of Corollary \ref{cor:bound}]
Since $f$ is tame, $p\nmid\deg_\xi f$ for any $\xi\in\mathsf{A}^1$. It follows that $\kappa=0$.  Thus Theorem \ref{thm:bound} implies the conclusion exception the equality in statement $2$. 

If $\mu(J(f)\cap\mathsf{H}^1)>0$ and $\log_p\mathrm{diam}(\cdot)$ is $\mu$-integrable on $J(f)\cap\mathsf{H}^1$, applying \eqref{equ:diam2} and the invariance of $\mu$, we conclude that 
\begin{equation*}
L(f,\mu)=\int_{\mathsf{P}^1}\log_p||f'|| d\mu=\int_{\mathsf{H}^1}\left(\log_p\mathrm{diam}(f(\xi))-\log_p\mathrm{diam}(\xi)\right)d\mu=0.
\end{equation*}
\end{proof}

\begin{proof}[Proof of Corollary \ref{cor:bound1}]
Since $f$ has no wandering Julia points in $\mathsf{H}^1$, by Lemma \ref{lem:repelling} and the invariance of $\mu$, we conclude that restricted on $J(f)\cap\mathsf{H}^1$, the measure $\mu$ is supported on periodic points. It follows that $\log_p\mathrm{diam}(\cdot)$ is $\mu$-integrable on $J(f)\cap\mathsf{H}^1$. Then the conclusion follows from Corollary \ref{cor:bound}.
\end{proof}

\section{Nonnegativeness of lower Lyapunov exponents}\label{sec:proof2}
In this section, we  prove Theorem \ref{thm:main} and Proposition \ref{prop:main}. Recall that $\mathcal{H}_c$ is the omitted tree of $f$ with respect to a critical point $c\in\mathrm{Crit}(f)$.

\begin{proof}[Proof of Theorem \ref{thm:main}]
For $n\ge 1$, set $w_n:=f^n(c)\in J(f)$ and consider the point $\xi_n\in\mathcal{H}_c$ such that $]w_n,\infty[\cap \mathcal{H}_c=[\xi_n,\infty[$. Then each $\xi_n$ is a vertex of $\mathcal{T}_f$ contained in the geometric sequence $\{G_\ell(w_n)\}_{\ell\ge 0}$ for $w_n$. Denoted by $\ell_n\ge 0$ the level such that $\xi_n=G_{\ell_n}(w_n)$. 

Note that $w_n\not\in\mathrm{Crit}(f)$ for all $n\ge 1$; for otherwise, since $c$ is the unique Julia critical point of $f$, there would exist $j\ge 1$ such that $w_j=c$, which would imply that $c\in\mathcal{F}(f)$. By Corollary \ref{coro:distance}, there exists $N:=N(n)\gg qn$ such that  for any $m\ge N$, 
\begin{equation}\label{equ:L-m}
\log_p|(f^n)'(w_1)|=\rho(\xi_f,G_m(w_1))-\rho(\xi_f,G_{m-qn}(w_n)).
\end{equation}
Observe that for sufficiently large $N$ with $N>\ell_n+qn$, we have $G_{m-qn}(w_n)\prec\xi_n$. It follows that
 \begin{equation}\label{equ:f-m}
 \rho(\xi_f,G_{n-qn}(w_n))=\rho(\xi_f,G_{\ell_n}(w_n))+\rho(G_{\ell_n}(w_n),G_{m-qn}(w_n)).
 \end{equation}
Since  $G_{\ell_n}(w_n)\in\mathcal{H}_c\setminus\ ]\xi_f,\infty[$, by Proposition \ref{lem:bound}, there exists $\alpha>0$ such that 
 \begin{equation}\label{equ:alpha}
 0\le \rho(\xi_f,G_{\ell_n}(w_n))\le\alpha.
 \end{equation} 
 Then from  \eqref{equ:L-m}, \eqref{equ:f-m} and \eqref{equ:alpha}, we conclude
 \begin{multline}\label{equ:der-bound}
 \left|\log_p|(f^n)'(w_1)|- \left(\rho(\xi_f,G_m(w_1))-\rho(G_{\ell_n}(w_n),G_{m-qn}(w_n))\right)\right|\\
 =\left|\rho(\xi_f,G_{\ell_n}(w_n))\right|\le\alpha.
 \end{multline}
 Moreover, applying \eqref{equ:L-difference}, we also obtain
 \begin{equation}\label{equ:L-m-}
 L_f^-(w_1)=\liminf_{n\to\infty}\frac{1}{n}\left(\rho(\xi_f,G_m(w_1))-\rho(G_{\ell_n}(w_n),G_{m-qn}(w_n))\right).
 \end{equation}
 
 Now we compare $\rho(\xi_f,G_m(w_1))$ and $\rho(\xi_n,G_{m-qn}(w_n))$. Consider the marked grid of $c$ and let $\gamma(j):=\gamma_{m,n}(j)$ be as in Section \ref{sec:grid}. Then for $\ell_n\le j\le m-qn-1$, since $]G_j(w_n),G_{j+1}(w_n)[\ \cap\mathcal{H}_c=\emptyset$, we have $]G_j(w_n),G_{j+1}(w_n)[\ \cap\ \mathcal{H}_c=\emptyset$, and hence 
 \begin{equation}\label{equ:not}
 ]G_{j+q\gamma(j)}(w_{n-\gamma(j)}), G_{j+1+q\gamma(j)}(w_{n-\gamma(j)})[\ \cap\ \mathcal{H}_c=\emptyset.
 \end{equation}
 Applying Lemma \ref{lem:deg constant}, Corollary \ref{coro:backward}, \eqref{equ:not}, and the tameness assumption, we obtain that for $\ell_n\le j\le m-qn-1$, 
 \begin{align*}
\rho(G_j(w_n),G_{j+1}(w_n))&=\deg_{c_0}f\cdot\rho(G_{j+q\gamma(j)}(w_{n-\gamma(j)}), G_{j+1+q\gamma(j)}(w_{n-\gamma(j)}))\\
&=\deg_{c_0}f\cdot\rho(G_{j+q\gamma(j)}(c), G_{j+1+q\gamma(j)}(c)).
\end{align*}
It follows that 
 \begin{align}\label{equ:<}
 \nonumber
 \rho(G_{\ell_n}(w_n),G_{m-qn}(w_n))&=\sum_{j=\ell_n}^{m-qn-1}\rho(G_j(w_n),G_{j+1}(w_n))\\ \nonumber
 &=\deg_{c}f\cdot\sum_{j=\ell_n}^{m-qn-1}\rho(G_{j+q\gamma(j)}(c), G_{j+1+q\gamma(j)}(c))\\ \nonumber
&=\sum_{j=\ell_n}^{m-qn-1}\rho(G_{j-1+q\gamma(j)}(w_1), G_{j+q\gamma(j)}(w_1))\\ 
 &\le\sum_{j=0}^{m-1}\rho(G_{j}(w_1), G_{j+1}(w_1))\\\nonumber
 &=\rho(\xi_f, G_m(w_1))
 \end{align}
 The inequality \eqref{equ:<} follows from the strictly increasing of  $j+q\gamma(j)$ (see Section \ref{sec:grid}); indeed,  
 $$1\le j+q\gamma(j)\le m-qn-1+q\gamma_{m,n}(m-qn-1)\le  m-qn-1+qn=m-1.$$
 By \eqref{equ:der-bound}, \eqref{equ:L-m-} and  \eqref{equ:<}, we conclude that  $\log_p|(f^n)'(w_1)|>-\alpha$ and $L_f^-(w_1)\ge 0$.
\end{proof}

With a refined version of the inequality  \eqref{equ:<}, we have the following criteria for positive $L_f^-(f(c))$.

\begin{corollary}\label{cor:>}
Under the assumptions in Theorem \ref{thm:main}, consider the geometric sequence $\{G_n(c)\}_{n\ge 0}$. If $\liminf_{n\to\infty}\rho(G_n(c),G_{n+1}(c))>0$, then $L^-_f(f(c))>0$.
\end{corollary}
\begin{proof}
Denote by $A:=\liminf_{n\to\infty}\rho(G_n(c),G_{n+1}(c))$. In \eqref{equ:<}, the term 
$$\sum_{j=\ell_n}^{m-qn-1}\rho(G_{j-1+q\gamma(j)}(w_1), G_{j+q\gamma(j)}(w_1))$$ 
is a summation of at most $m-n$ terms. 
Thus there exists $\epsilon>0$ such that 
$$\sum_{j=0}^{m-1}\rho(G_{j}(w_1), G_{j+1}(w_1))-\sum_{j=\ell_n}^{m-qn-1}\rho(G_{j-1+q\gamma(j)}(w_1), G_{j+q\gamma(j)}(w_1))\ge n\epsilon A.$$
Then by \eqref{equ:L-m-}, 
we conclude $L_f^-(f(c))\ge \epsilon A>0$.
\end{proof}

\begin{proof}[Proof of Proposition \ref{prop:main}]
Since $c$ is not in the basin of an attracting cycle, by Lemma \ref{lem:nonwandering}, the critical point  $c$ is eventually mapped to an indifferent periodic domain or a wandering but not a strictly wandering domain. Since $f$ has finitely many critical points in $K$, there exists a smallest $\ell_0\ge 1$ such that forward orbit of the Fatou component containing $w_{\ell_0}$ is disjoint with $\mathrm{Crit}(f)$. To show $L_f(w_\ell)=0$ for any $\ell\ge\ell_0$, it suffices to show that $L_f(w_{\ell_0})=0$

Let $j\ge\ell_0$ be the smallest integer such that the boundary, denoted by $\xi\in\mathsf{H}^1$, of the Fatou component containing $w_j$ is a periodic point. Let $m\ge 1$ be the period of $\xi$ and set 
$$r_{\min}:=\min\{\mathrm{diam}(f^i(\xi)): 0\le i \le m-1\}\ \text{and}\ r_{\max}:=\max\{\mathrm{diam}(f^i(\xi)): 0\le i\le m-1\}.$$
Observing that the choice of $\ell_0$ implies that $f$ is one-to-one on the Fatou component containing $w_\ell$ for $\ell\ge\ell_0$,  by Lemma \ref{lem:basic}(2), we conclude that for all $n\ge 1$,
$$\frac{r_{\min}}{r_{\max}}\le |(f^n)'(w_j)|\le\frac{r_{\max}}{r_{\min}}.$$
It follows that
$$L_f(w_{\ell_0})=\lim_{n\to\infty}\frac{1}{n}\log_p|(f^n)'(w_{\ell_0})|=\lim_{n\to\infty}\frac{1}{n}\log_p|(f^{n-(j-\ell_0)})'(w_j)\cdot(f^{j-\ell_0})'(w_{\ell_0})|=0.$$
\end{proof}


\section{Examples}\label{sec:ex}

In this section, we give provide examples about the lower Lyapunov exponents at critical values. 

Inspired by \cite[Section 5]{Przytycki99}, in our first example, we show that there exists a tame quartic polynomial in $\mathbb{C}_p[z]$ having two Julia critical points and no critical relation such that the lower Lyapunov exponent at a critical value is $-\infty$

\begin{example}\label{ex:2}
Let $p\ge 5$ be a prime. Consider the quartic polynomials $f\in\mathbb{C}_p$ such that $f$ has a critical point whose critical value is a repelling fixed point. Then $f$ has the form of 
$$f(z)=a_4(z-c)^4+a_3(z-c)^3+a_2(z-c)^2+a_0$$
with $f(a_0)=a_0$ and $|f'(a_0)|>1$, where $a_i$ and $c$ are in $\mathbb{C}_p$. We can further set $f'(a_0)=p^{-1}$ for connivence. Then such polynomials form a subset $W$ of $\mathbb{C}_p^5$ satisfying two polynomial equations.

Observes that $c\in J(f)$. Then the backward orbit of $0$ is dense in $J(f)$ \cite[Proposition 5.23]{Benedetto19}. For any $0<r<1$, consider the open disk $D(0,r)\subset\mathbb{C}_p$ and set $S:=S_{r,f}=\overline{D}(0,r)\cap J(f)$. Now for $n\ge 1$, set
$$X_n:=\{z\in S: |(f^{n})'(z)|<p^{-p^n}\}.$$
Observes that $0\in\left(\cup_{i=0}^{n-1}f^{-i}(0)\right)\cap S\subset X_n$. For $n\ge 2$, denoted by $Y_n:=\cup_{i=0}^{n-2}f^{-i}(0)$ and set 
$$Z_n:=X_n\setminus Y_n.$$
It follows that $f^{-(n-1)}(0)\cap S\subset Z_n$.  For any fixed $N\ge 2$, we conclude that $\cup_{n\ge N}Z_n$ is dense in $S$. Moreover, $Z_n$ is open, so is $\cup_{n\ge N}Z_n$. 
 By Baire category theorem, the intersection $Z:=\bigcap_{N\ge 1}\bigcup_{n\ge N}Z_n$ is dense in $S$.  Pick a point $w\in Z$ and observe that $w$ in not in the backward orbit of $c$. Then we have 
 $$L_f^{-}(w)=\liminf_{n\to\infty}\frac{1}{n}\log_p|(f^n)'(w)|\le\liminf_{n\to\infty}\frac{-p^n}{n}=-\infty$$

Denote by $c'$ another critical point of $f$ with $c'\not=c$ and remember that $w$ is dependent on $f$. Now we show there are suitable $f$ and $w$ such that $f(c')=w$. Consider the map $F: W\to\mathbb{C}_p$, sending $f$ to $f(c')-c$. Then $F$ is analytic and there is $f_0\in W$ such that $f(c')-c=0$. It follows that there image $F(W)$ contains a neighborhood of $0$, say $D(0,r')$. Let the $r$ above satisfy $0<r<r_1$ and considering the corresponding $w-c\in D(0,r)$. Then there exists an $f\in W$ such that $F(f)=w-c$. Thus for such an $f$, we have $f(c')=w$, and hence $L_f^-(f(c'))=-\infty$.
\end{example}

In our next example, we give a cubic tame polynomial satisfying the assumptions in Theorem \ref{thm:main} and having lower Lyapunov exponent $0$ at the unique critical value in the Julia set. 

\begin{example}\label{ex:0}
Recall that $\mathbb{L}$ is the completion of the Puiseux series over $\mathbb{C}$. Fix $\alpha\in\mathbb{L}$ with $|\alpha|>1$ and consider a cubic polynomial $f_{v}(z)=\alpha^2(z-1)^2(z-2)+v$ with $v\in\mathbb{L}[z]$. By \cite[Theorem 5.3(iii)]{Kiwi06}, for the parameters $v$ in the boundary of shift locus, the critical point $c=1$ of $f_v$ is contained in the filled Julia set $\mathcal{K}(f_v)$ and the marked grid $\mathbf{M}(c)=(M_{\ell,k}(c))$ is aperiodic. Now we consider a parameter $v\in\mathbb{L}$ such that $\mathbf{M}(c)$ satisfies the following Fibonacci property: $M_{\ell_n,k_n}(c)=1$ for $\ell_1=1, k_1=2,  k_2=3, \ell_{n+1}=\ell_n+k_n, k_{n+1}=k_n+k_{n-1}$ and all the other critical entries are determined by the admissible conditions in \cite[Proposition 4.5 and Definition 5.2]{Kiwi06}. The existence of such a parameter $v$ follows from \cite[Theorem 5.3(ii)]{Kiwi06}. For $n\ge 1$, considering the dynamical sequences $\{L_n(c)\}_{n\ge 0}$ of $c$ and  $\{L_n(v)\}_{n\ge 0}$ of $v$, since $f_v^{k_n-1}$ maps the closed disk in $\mathbb{L}$ corresponding to $L_{\ell_{n+1}-1}(v)$ bijectively to the closed disk in $\mathbb{L}$ corresponding to $L_{\ell_n}(c)=L_{\ell_{n+1}-k_n}(f^{k_n-1}(v))$, we conclude that 
\begin{align}\label{equ:length}
2\rho(L_{\ell_{n+1}}(c),\xi_{f_v})=\rho(L_{\ell_{n}}(c),\xi_{f_v})+\rho(\xi_{f_v},f^{k_n}(\xi_{f_v})),  \ \ \ \text{and}
\end{align}
\begin{align}\label{equ:der-lya}
\left|(f_v^{k_n-1})'(v)\right|=\frac{\mathrm{diam}(L_{\ell_n}(c))}{\mathrm{diam}(L_{\ell_{n+1}-1}(v))}.
\end{align}
It follows from \eqref{equ:length} that $\rho(L_{\ell_n}(c),\xi_{f_v})\to\infty$, as $n\to\infty$, and hence $c\in J(f_v)$. In what follows, we estimate $L^-_{f_v}(v)$ using \eqref{equ:der-lya}.

Note that $f_v$ maps $[L_{\ell_{n+1}}(c)), \xi_{f_v}[$ onto $[L_{\ell_{n+1}-1}(v),f_v(\xi_{f_v})[$ with local degree 2. Thus 
$$\log\frac{\mathrm{diam}(f_v(\xi_{f_v}))}{\mathrm{diam}(L_{\ell_{n+1}-1}(v))}=2\log\frac{\mathrm{diam}(\xi_{f_v})}{\mathrm{diam}(L_{\ell_{n+1}}(c))},$$
that is 
$$\log\mathrm{diam}(L_{\ell_{n+1}-1}(v))=2\log\mathrm{diam}(L_{\ell_{n+1}}(c))-2\log\mathrm{diam}(\xi_{f_v})+\log\mathrm{diam}(f_v(\xi_{f_v})).$$
Thus
we conclude that as $n\to\infty$, 
\begin{equation}\label{equ:est}
\frac{\log\left|(f_v^{k_n-1})'(v)\right|}{k_n-1}=\frac{1}{k_n-1}\left(2\log\frac{\mathrm{diam}(\xi_{f_v})}{\mathrm{diam}(L_{\ell_{n+1}}(c))}-\log\frac{\mathrm{diam}(\xi_{f_v})}{\mathrm{diam}(L_{\ell_n}(c))}\right)+o(1).
\end{equation}
For any $j\ge 2$, noting that 
$$\ell_{j+1}-\ell_{j}=k_j=k_{j-1}+k_{j-2}=(\ell_{j}-\ell_{j-1})+(\ell_{j-1}-\ell_{j-2})=\ell_{j}-\ell_{j-2}$$ 
and considering the map $f_v^{k_j}$, we have that 
$$2\rho(L_{\ell_{j+1}}(c),L_{\ell_j}(c))=\rho(L_{\ell_j}(c),L_{\ell_{j-2}}(c))=\rho(L_{\ell_j}(c),L_{\ell_{j-1}}(c))+\rho(L_{\ell_{n-1}}(c),L_{\ell_{j-2}}(c))$$
Then there exists $C>0$ such that for any $n\ge 1$,
\begin{equation}\label{equ:C}
1<\rho(L_{\ell_{n+1}}(c),L_{\ell_n}(c))=\frac{\mathrm{diam}(L_{\ell_n}(c))}{\mathrm{diam}(L_{\ell_{n+1}}(c))}<C.
\end{equation}
Thus  by \eqref{equ:est} and \eqref{equ:C}, as $n\to\infty$,
\begin{multline*}
\frac{\log\left|(f_v^{k_n-1})'(v)\right|}{k_n-1}=\frac{1}{k_n-1}\log\frac{1}{\mathrm{diam}(L_{\ell_{n+1}}(c))}+o(1)\\
=\frac{1}{k_n-1}\sum_{j=0}^{n}\log\frac{\mathrm{diam}(L_{\ell_j}(c))}{\mathrm{diam}(L_{\ell_{j+1}}(c))}+o(1).
\end{multline*}
By \eqref{equ:C} again, we obtain that as $n\to\infty$,
$$0\le \frac{\log\left|(f_v^{k_n-1})'(v)\right|}{k_n-1}\le\frac{Cn}{k_n-1}+o(1).$$
Thus 
$$\frac{\log\left|(f_v^{k_n-1})'(v)\right|}{k_n-1}\to 0.$$
Hence $L_{f_v}^-(v)\le 0$. By Theorem \ref{thm:main}, we in fact have $L_{f_v}^-(v)=0$.
\end{example}

\bibliographystyle{siam}
\bibliography{references}

\end{document}